\documentclass[12pt]{amsart}

\usepackage{stmaryrd}
\usepackage[hidelinks]{hyperref}
\usepackage{subcaption}
\usepackage{tikz}
\usetikzlibrary{shapes,fit,positioning,calc,matrix,arrows}
\tikzset{
  optree/.style={scale=.5,thick,grow'=up,level distance=10mm,inner sep=1pt},
  comp/.style={draw=none,circle,fill,line width=0,inner sep=0pt},
  dot/.style={draw,circle,fill,inner sep=0pt,minimum width=3pt},
  circ/.style={draw,circle,inner sep=1pt,minimum width=4mm},
  emptycirc/.style={draw,circle,inner sep=1pt,minimum width=2mm},  
  root/.style={level distance=10mm,inner sep=1pt},
  leaf/.style={draw=none,circle,fill,line width=0,inner sep=0pt},
  nodot/.style={draw,circle,inner sep=1pt},
}
\usepackage{fullpage}
\usepackage{amssymb}
\usepackage[shortcuts]{extdash}

\newtheorem*{thm*}{Theorem}

\newtheorem{thm}{Theorem}
\newtheorem{lem}[thm]{Lemma}
\newtheorem{prp}[thm]{Proposition}
\newtheorem{crl}[thm]{Corollary}

\theoremstyle{definition}
\newtheorem{dfn}{Definition}

\theoremstyle{remark}
\newtheorem{rmk}{Remark}

\DeclareMathOperator{\rev}{rev_t}

\newcommand{\oeis}[1]{\href{https://oeis.org/#1}{{#1}}}

\newcommand{\K}{\mathbb{K}}
\newcommand{\N}{\mathbb{N}}
\newcommand{\Sym}{\mathfrak{S}}
\newcommand{\g}{\mathfrak{g}}
\newcommand{\Mult}{\mathrm{Mult}}
\newcommand{\Part}{\mathrm{Part}}
\newcommand{\descr}{\mathrm{descr}}
\newcommand{\dg}{\mathrm{dg}}
\newcommand{\End}{\mathrm{End}}
\newcommand{\Hom}{\mathrm{Hom}}

\newcommand{\Red}{\mathcal{FRG}}
\newcommand{\V}{\mathcal{V}}
\newcommand{\E}{\mathcal{E}}
\newcommand{\R}{\mathcal{R}}
\newcommand{\RG}{\mathcal{RG}}
\newcommand{\FRGHT}{\mathcal{FG}}
\newcommand{\FRHT}{\mathcal{FH}}
\newcommand{\RHT}{\mathcal{H}}
\newcommand{\RGT}{\mathcal{G}}
\newcommand{\RT}{\mathcal{RT}}
\newcommand{\Tt}{\mathcal{T}}
\newcommand{\Ss}{\mathcal{S}}
\newcommand{\Pp}{\mathcal{P}}
\newcommand{\F}{\mathcal{F}}

\newcommand{\MC}{\mathrm{MC}}
\newcommand{\Tw}{\mathrm{Tw}}

\newcommand{\FinSet}{\mathrm{BijSet}}
\newcommand{\Vect}{\mathrm{Vect}}
\newcommand{\RCG}{\mathrm{RedComGreg}}
\newcommand{\CG}{\mathrm{ComGreg}}
\newcommand{\greg}{\mathrm{Greg}}
\newcommand{\Lie}{\mathrm{Lie}}
\newcommand{\uCom}{\mathrm{uCom}}
\newcommand{\PL}{\mathrm{PreLie}}

\newcommand{\CPL}{\mathrm{ComPreLie}}
\newcommand{\FMan}{\mathrm{FMan}}

\usepackage{lipsum}

\usepackage[backend=bibtex,isbn=false]{biblatex}
\begin{document}

\title{Hypertrees and embedding of the $\mathrm{FMan}$ operad}

\author[]{Paul Laubie} 
\thanks{Paul Laubie was partially supported by the French national research agency [grant ANR-20-CE40-0016].}
\address{Institut de Recherche Mathématique Avancée, UMR 7501, Université de Strasbourg et CNRS, 7 rue René-Descartes, 67000 Strasbourg CEDEX, France}
\email{laubie@unistra.fr}

\begin{abstract}
    The operad $\mathrm{FMan}$ encodes the algebraic structure on vector fields of Frobenius manifolds, in the same way as the operad $\mathrm{Lie}$ encodes the algebraic structure on vector fields of a smooth manifold. It is well known that the operad $\mathrm{Lie}$ admits an embedding in the operad $\mathrm{PreLie}$ encoding pre-Lie algebras. We prove a conjecture of Dotsenko stating that the operad $\mathrm{FMan}$ admits an embedding in the operad $\mathrm{ComPreLie}$. The operad $\mathrm{ComPreLie}$ is the operad encoding pre-Lie algebras with an additional commutative product such that right pre-Lie multiplications act as derivations. To prove this result, we first remark a link between the Greg trees and the so-called operadic twisting of $\mathrm{PreLie}$. We then give a combinatorial description of the operad $\mathrm{ComPreLie}$ \emph{à la} Chapoton-Livernet with forests of rooted hypertrees. We generalize this construction to forests of rooted Greg hypertrees, and then use operadic twisting techniques to prove the conjecture.
\end{abstract}

\maketitle

\section*{Introduction}

Pre-Lie algebras (or left--symmetric algebras) appear in a wide range of mathematical domains. Indeed, they appear in combinatorics, deformation theory, differential geometry, renormalization, see \cite{LSA}. Moreover, the operad $\PL$ encoding pre-Lie algebras is deeply related to a famous combinatorial species, the rooted trees. Indeed, an operadic structure is defined on the species of rooted trees in \cite{RT}, and the operad obtained this way is proven to be isomorphic to $\PL$. This construction can be generalized to the species of rooted Greg trees which are rooted trees with white and black vertices such that black vertices are unlabeled and have at least two children, see \cite{Me}, defining the operad $\greg$. We show that the operad $\greg$ is closely related to the so-called operadic twisting of $\PL$ \cite{PLTW,TW}. From this fact, one may wonder if there are analogous combinatorial descriptions of some related operads, or conversely if some combinatorial objects related to rooted trees admit operadic structures analogous to the one on rooted trees. The operad $\CPL$ is the operad encoding pre-Lie algebras with an additional commutative product such that right pre-Lie multiplications act as derivations. This operad was first introduced in \cite{DefCPL}. We show that it admits a combinatorial description in terms of forests of rooted hypertrees which can be further generalized to the species of forests of rooted Greg hypertrees, thus defining a new operad that we denote $\CG$. From the link between $\greg$ and the operadic twisting of $\PL$, we relate $\CG$ to the operadic twisting of $\CPL$ to prove a conjecture of Dotsenko \cite{conjFMan} stating that the operad $\FMan$ encoding the algebraic structure on vector fields of Frobenius manifolds admits an embedding in the operad $\CPL$. To do so, we need to slightly modify the operad $\CG$ to obtain the operad $\dg\CG$, and we relate $\dg\CG$ to the operadic twisting of $\CPL$, allowing us to show the main theorem of this paper:
\begin{thm*} [Th.~\ref{thm:main}]
    The morphism of operads $\FMan\to\CPL$ is injective.
\end{thm*}
This theorem comes with a side result:
\begin{thm*} [Th.~\ref{thm:CPLK}]
    The operad $\CPL$ is Koszul.
\end{thm*}

\paragraph*{\textbf{Organization of the paper}} In the first section, we recall general fact on species and operads. We then recall the construction of the operadic structure on the species of rooted trees from \cite{RT}, which give the operad $\PL$, and its generalization on the rooted Greg trees from \cite{Me}, which give the operad $\greg$. In Section~\ref{sec:greg}, give a small introduction to the operadic twisting, see \cite{PLTW,TW}, and we relate the operad $\greg$ to the operadic twisting of $\PL$. In Section~\ref{sec:FRHT}, we first show that the operad $\CPL$ is Koszul. Then, we extend the operadic structure from the species of rooted trees to the species of forests of rooted hypertrees, and using the Koszulness of $\CPL$, we show that the underlying species of $\CPL$ is the species of forests of rooted hypertrees. In the next section, we use the same idea extending the operadic structure from the species of rooted Greg trees and from the forests of rooted hypertrees to the species of forests of rooted Greg hypertrees, obtaining the operad $\CG$. We show that this operad is Koszul, and we give a quadratic presentation of this operad. However, the operad $\CG$ cannot directly be used to prove the main theorem. In Section~\ref{sec:red}, we explain how to modify the operad $\CG$. In order to do so, we introduce the species of reduced rooted Greg hypertrees and put a structure of differential graded operad on this species, defining this way the differential graded operad $\dg\CG$. In the final section, we use the operad $\dg\CG$ to prove the main theorem. Namely, we show that the cohomology of $\dg\CG$ is concentrated in degree $0$, and we prove that it is in fact isomorphic the operad $\FMan$, proving the main theorem.

\section{Rooted trees and Rooted Greg trees}

Let $\K$ be a field of characteristic $0$. Let $\underline{n}$ be the set $\{1,\dots,n\}$. We recall the bases of the theory of species and operads. We refer to \cite{LinSpe} for a more detailed introduction to the theory of species and to \cite{AlgOp} for the theory of operads. In the following, all species will be linear species and operads will be symmetric operads. One need to be careful since the objects we are calling linear species are ``species with values in $\Vect_\K$'' which the same convention as in \cite{AlgOp} but different from the one in \cite{LinSpe} where linear species correspond to ``linearly ordered species'' what we will not use here.

Let $\FinSet$ be the category such that the objects are the finite sets and the morphisms are the bijections. A \emph{linear species} $\Ss$ is a functor from $\FinSet$ to $\Vect_\K$ the category of vector spaces over $\K$. Its \emph{Hilbert series} is the formal power series defined by:
$$ f_\Ss(t) = \sum_{n\in\N} \frac{\dim_\K(\Ss(\underline{n}))}{n!} t^n $$
For $\Ss$ a species, the \emph{Schur functor} $\F_\Ss:\Vect_\K\to\Vect_\K$ is the functor defined by: 
$$\F_\Ss(V) = \bigoplus_{n\in\N} \Ss(\underline{n})\otimes_{\Sym_n} V^{\otimes n}$$
where $V^{\otimes n}$ is the $n$-th tensor power of $V$ and $\Sym_n$ acts on $V^{\otimes n}$ by permuting the factors. The species $\Ss$ is uniquely determined by the functor $\F_\Ss$.
For $\Ss$ and $\R$ two species, there exists a unique species $\Ss\circ\R$ such that $F_\Ss\circ F_\R = F_{\Ss\circ\R}$. The operation $\circ$ is called the \emph{plethysm} of species. It gives a structure of monoidal category on the category of species. Moreover, $f_{\Ss\circ\R} = f_\Ss(f_\R)$. Another monoidal structure of the category of species that will be used latter in this article is given the Cauchy product of species. The Cauchy product of $\Ss$ and $\R$ is defined by:
$$ (\Ss\otimes\R)(A) = \bigoplus_{A=A_1\sqcup A_2} \Ss(A_1)\otimes\R(A_2) ,$$
where $A_1\sqcup A_2$ is the disjoint union of $A_1$ and $A_2$. We recall the notation of the usual species $X$ the \emph{singleton species} such that $X(\underline{1})=\K$ endowed with the trivial action and $X(\underline{n})=0$ for $n\neq 1$, and the species $E$ the \emph{set species} such that $E(\underline{n})=\K$ endowed with the trivial action for all $n\in\N$. For $k\in\N$, the species $E_{\geq k}$ is the species such that $E_{\geq k}(\underline{n})=\K$ endowed with the trivial action if $n\geq k$ and $E_{\geq k}(\underline{n})=0$ otherwise.

An \emph{operad} $\Pp$ is a monoid object in the category of species with respect to the plethysm. It can equivalently be defined as a species $\Ss$ together with a collection of equivariant maps, the \emph{partial compositions}:
$$\circ_i:\Ss(A\sqcup\{i\})\otimes\Ss(B)\to\Ss(A\sqcup B)$$
for all finite sets $A$ and $B$, verifying the \emph{parallel} and \emph{sequential} axioms, see \cite[Paragraph~5.3.7]{AlgOp}.
The parallel and sequential axioms ensure that the partial compositions behave like partial compositions of multilinear maps. Indeed, let $V$ be a vector space, and let us define $\End_V(n)=\Hom(V^{\otimes n},V)$, then, the partial compositions give a structure of operad on $\End_V$. The operad $\End_V$ is called the \emph{endomorphism operad} of $V$. Moreover, for $\Pp$ an operad, a morphism of operads $\Pp\to\End_V$ is a \emph{$\Pp$--algebra structure} on $V$. If $\Pp$ is defined by generators and relations, the data of a morphism of operad $\Pp\to\End_V$ is the same as a collection of multilinear maps of $V$ verifying the relations defining $V$. 

Those definitions generalize in a straightforward way to the graded and the differential graded cases using Koszul sign rules in the latter case. We will be using operadic Koszul theory, see \cite{KOP,AlgOp}, to compute Hilbert series of operads, using the following result:

\begin{prp}\label{prp:KoszulHilbertSeries}
    Let $\Pp$ be a Koszul operad and $\Pp^!$ be its Koszul dual. Let $f_\Pp$ and $f_{\Pp^!}$ be the Hilbert series of $\Pp$ and $\Pp^!$ respectively. Then, we have the following equality:
    $$ f_\Pp(-f_{\Pp^!}(-t)) = t $$
\end{prp}

As a general tool to show Koszulness of operads, we will use the theory of Gröbner bases for operads. We refer to \cite{GB} for a general introduction to the theory of shuffle operads and Gröbner bases for operads. We will in fact use the more general language of rewriting systems. A \emph{rewriting system} is a set of \emph{writing rules} which are pairs $(a,b)$. Let us write $x\mapsto y$ to denote a \emph{rewriting step} from $x$ to $y$, it corresponds to the application of a writing rule $(a,b)$ to $x$ such that $a$ is a sub-object of $x$ wich is replaced by $b$ to obtain $y$. In the case of operads, we will rewrite linear combinations of \emph{shuffle trees}. Since this is not the main point of this article, we will not give the definition of shuffle trees, we refer to \cite{GB} for a precise definition. Let us write $x\overset{*}{\mapsto} y$ for a finite sequence of rewriting steps from $x$ to $y$. A rewriting system is \emph{terminating} if any sequence of rewriting steps is finite. A rewriting system is \emph{locally confluent} if $y\mapsfrom x\mapsto z$ imply that we have $t$ such that $y\overset{*}{\mapsto}t \overset{*}{\mapsfrom} z$. A rewriting system is \emph{convergent} if it is terminating and locally confluent. Any Gröbner basis gives raise to a convergent rewriting system, moreover, we have a one-to-one correspondence between elements of the Gröbner basis and rewriting rules of the rewriting system. We will use the following result to show Koszulness of operads:

\begin{prp}\label{prp:QGB}
    Let $\Pp$ be an operad. If $\Pp$ admits a quadratic Gröbner basis, then $\Pp$ is Koszul. More generally, if $\Pp$ admits quadratic convergent rewriting system, then $\Pp$ is Koszul.
\end{prp}

Let $\Tt$ be the free operad functor. The operad $\Lie$ is defined by the following presentation:
$$ \Tt\{l\}/\langle l\circ_1 l - l\circ_2 l - (l\circ_1 l).(2\;3)\rangle $$
where $l$ is a binary operation, and the action of $\Sym_2$ on $\K\langle l\rangle$ is given by $l.(1\; 2)=-l$. The operad $\PL$ is defined by the following presentation:
$$ \Tt\{x,y\}/\langle (x\circ_1 x - x\circ_2 x) - (x\circ_1 x - x\circ_2 x).(2\;3)\rangle $$
where $x$ and $y$ are binary operations, and the action of $\Sym_2$ on $\K\langle x,y\rangle$ is given by $x.(1\; 2)=y$. We recall that we have the following injective morphism of operads $\Lie\to\PL$ given by $l\mapsto x-y$.

\begin{dfn}
    A \emph{rooted tree} $\tau$ is a finite graph $(\V,\E)$ without cycles with a distinguished vertex called the \emph{root}. A \emph{labeling} on $\tau$ is a bijective map $l:\V\to L$ with $L$ a set of labels. The species of rooted trees is the species $\RT$ such that $\RT(L)$ is the set of rooted trees labeled by $L$. In the following, rooted trees will always be labeled. Rooted trees will be represented with the root at the bottom.
\end{dfn}

Let us recall the construction of \cite{RT} of an operad structure on the species of rooted trees.

\begin{dfn}
    Let $S$ and $T$ be two rooted trees and $i$ be a label of a vertex $v_i$ of $S$. Let $B$ be the tree bellow the vertex $v_i$ in $S$, and $C=\{C_1,\dots,C_n\}$ be the set of children of $v_i$ in $S$. 
    \[
        \vcenter{\hbox{\begin{tikzpicture}[scale=0.5]     
            \node (S) at (0, 0) {};
            \node[isosceles triangle, isosceles triangle apex angle=60, draw, rotate=270, fill=white!50, minimum size =24pt] (ST) at (S){};
            \draw[circle, fill=black] (ST.east) circle [radius=2pt];
            \node at (S) {\scalebox{0.8}{$S$}};
        \end{tikzpicture}}}
        \qquad
        =
        \qquad
        \vcenter{\hbox{\begin{tikzpicture}[scale=0.5]     
            \node (lv) at (0, 0) {};
            \node (SD) at (0, -3) {};
            \node (SU1) at (-2, 3) {};
            \node (dot) at (0, 3) {};
            \node (SUk) at (2, 3) {};
            \node[isosceles triangle, isosceles triangle apex angle=60, draw, rotate=-90, fill=white!50, minimum size =24pt] (SDT) at (SD){};
            \node[isosceles triangle, isosceles triangle apex angle=60, draw, rotate=270, fill=white!50, minimum size =24pt] (SU1T) at (SU1){};
            \node[isosceles triangle, isosceles triangle apex angle=60, draw, rotate=270, fill=white!50, minimum size =24pt] (SUkT) at (SUk){};
            \draw[thick] (SU1T.east)--(lv);
            \draw[thick] (SUkT.east)--(lv);
            \draw[thick] (lv)--(SDT.west);
            \node at (dot) {$\cdots$};
            \draw[circle, fill=white] (lv) circle [radius=24pt];
            \draw[circle, fill=black] (SDT.east) circle [radius=2pt];
            \draw[circle, fill=black] (SU1T.east) circle [radius=2pt];
            \draw[circle, fill=black] (SUkT.east) circle [radius=2pt];
            \node at (lv) {\scalebox{1}{$v_i$}};
            \node at (SU1) {\scalebox{0.8}{$C_1$}};
            \node at (SUk) {\scalebox{0.8}{$C_n$}};
            \node at (SD) {\scalebox{0.8}{$B$}};
        \end{tikzpicture}}}
        \qquad
        =
        \qquad
        \vcenter{\hbox{\begin{tikzpicture}[scale=0.5]     
            \node (lv) at (0, 0) {};
            \node (SD) at (0, -2.7) {};
            \node (F) at (0, 3) {};
            \node[isosceles triangle, isosceles triangle apex angle=60, draw, rotate=270, fill=white!50, minimum size =24pt] (SDT) at (SD){};
            \node[isosceles triangle, isosceles triangle apex angle=60, draw, rotate=270, fill=white!50, minimum size =24pt] (FT) at (F){};
            \draw[thick, double] (FT.east)--(lv);
            \draw[thick] (lv)--(SDT.west);
            \draw[circle, fill=white] (lv) circle [radius=24pt];
            \draw[circle, fill=black] (FT.east) circle [radius=2pt];
            \draw[circle, fill=black] (SDT.east) circle [radius=2pt];
            \node at (lv) {\scalebox{1}{$v_i$}};
            \node at (F) {\scalebox{0.8}{$C$}};
            \node at (SD) {\scalebox{0.8}{$B$}};
        \end{tikzpicture}}}
    \]
    We use circles to represent vertices, triangles to represent trees or sets of trees and double edges to represent that each tree of the set $C$ is grafted at $v_i$. The \emph{insertion} of $T$ in $S$ at the vertex $i$ denoted $S\circ_i T$ is the formal sum of all possible way to graft the rooted trees $C_1,\dots,C_n$ on vertices of $T$ and then grafting the result on the parent of $v_i$ in $B$.
\end{dfn}

Let us compute the following example:
\[
    \vcenter{\hbox{\begin{tikzpicture}[scale=0.25]     
        \node (3) at (0, 0) {};
        \node (1) at (0, -3) {};
        \draw[thick] (1)--(3);
        \draw[circle, fill=white] (1) circle [radius=24pt];
        \draw[circle, fill=white] (3) circle [radius=24pt];
        \node at (1) {\scalebox{1}{$1$}};
        \node at (3) {\scalebox{1}{$3$}};
    \end{tikzpicture}}} \; \circ_1 \; \vcenter{\hbox{\begin{tikzpicture}[scale=0.25]     
        \node (2) at (0, 0) {};
        \node (1) at (0, -3) {};
        \draw[thick] (1)--(2);
        \draw[circle, fill=white] (1) circle [radius=24pt];
        \draw[circle, fill=white] (2) circle [radius=24pt];
        \node at (1) {\scalebox{1}{$1$}};
        \node at (2) {\scalebox{1}{$2$}};
    \end{tikzpicture}}} \; = \; \vcenter{\hbox{\begin{tikzpicture}[scale=0.25]     
        \node (2) at (0, 0) {};
        \node (1) at (0, -3) {};
        \node (3) at (0, 3) {};
        \draw[thick] (1)--(2);
        \draw[thick] (2)--(3);
        \draw[circle, fill=white] (1) circle [radius=24pt];
        \draw[circle, fill=white] (2) circle [radius=24pt];
        \draw[circle, fill=white] (3) circle [radius=24pt];
        \node at (1) {\scalebox{1}{$1$}};
        \node at (2) {\scalebox{1}{$2$}};
        \node at (3) {\scalebox{1}{$3$}};
    \end{tikzpicture}}} \; + \; \vcenter{\hbox{\begin{tikzpicture}[scale=0.25]     
        \node (2) at (-1.5, 0) {};
        \node (1) at (0, -3) {};
        \node (3) at (1.5, 0) {};
        \draw[thick] (1)--(3);
        \draw[thick] (2)--(1);
        \draw[circle, fill=white] (1) circle [radius=24pt];
        \draw[circle, fill=white] (2) circle [radius=24pt];
        \draw[circle, fill=white] (3) circle [radius=24pt];
        \node at (1) {\scalebox{1}{$1$}};
        \node at (2) {\scalebox{1}{$2$}};
        \node at (3) {\scalebox{1}{$3$}};
    \end{tikzpicture}}}
\]

\begin{prp}\cite{RT}
    The insertions satisfy the parallel and sequential axioms. Hence, they give a structure of operad on the species of rooted trees.
\end{prp}

\begin{thm}\cite[Theorem 1.9]{RT}
    The operad $(\RT,\{\circ_i\})$ is isomorphic to $\PL$. Moreover, the isomorphism is given by $ \vcenter{\hbox{\begin{tikzpicture}[scale=0.25]     
        \node (2) at (0, 0) {};
        \node (1) at (0, -3) {};
        \draw[thick] (1)--(2);
        \draw[circle, fill=white] (1) circle [radius=24pt];
        \draw[circle, fill=white] (2) circle [radius=24pt];
        \node at (1) {\scalebox{1}{$1$}};
        \node at (2) {\scalebox{1}{$2$}};
    \end{tikzpicture}}} \mapsto x$.
\end{thm}

Let us state a generalization of this construction. We first need to define a generalization of the rooted trees, namely the rooted Greg trees \cite{Me}.

\begin{dfn}
    A \emph{rooted Greg tree} $\tau$ is a rooted tree with black and white vertices such that the white vertices are labeled, and the black vertices are unlabeled and have at least two children. The \emph{weight} of $\tau$ is the number of black vertices of $\tau$. The species $\RGT$ of rooted Greg trees is the species such that $\RGT(L)$ is the set of rooted Greg trees labeled by $L$. The species $\RGT$ is graded by the weight, in particular $\RGT_0$ is the species of rooted trees. The set $\RGT(\underline{2})$ is depicted in Figure \ref{fig:RGT2}.
\end{dfn}

\begin{figure}[ht]
    \caption{The set $\RGT(\underline{2})$}
    \label{fig:RGT2}
    \[
        \vcenter{\hbox{\begin{tikzpicture}[scale=0.25]
            \draw[thick] (0, 0)--(0, 3);
            \draw[fill=white, thick] (0, 0) circle [radius=24pt];
            \draw[fill=white, thick] (0, 3) circle [radius=24pt];
            \node at (0, 0) {\scalebox{1}{1}};
            \node at (0, 3) {\scalebox{1}{2}};
        \end{tikzpicture}}}
        \quad,\quad
        \vcenter{\hbox{\begin{tikzpicture}[scale=0.25]
            \draw[thick] (0, 0)--(0, 3);
            \draw[fill=white, thick] (0, 0) circle [radius=24pt];
            \draw[fill=white, thick] (0, 3) circle [radius=24pt];
            \node at (0, 0) {\scalebox{1}{2}};
            \node at (0, 3) {\scalebox{1}{1}};
        \end{tikzpicture}}}
        \quad,\quad
        \vcenter{\hbox{\begin{tikzpicture}[scale=0.25]
            \draw[thick] (0, 0)--(-1.5, 3);
            \draw[thick] (0, 0)--(1.5, 3);
            \draw[fill=black, thick] (0, 0) circle [radius=24pt];
            \draw[fill=white, thick] (-1.5, 3) circle [radius=24pt];
            \draw[fill=white, thick] (1.5, 3) circle [radius=24pt];
            \node at (0, 0) { };
            \node at (-1.5, 3) {\scalebox{1}{1}};
            \node at (1.5, 3) {\scalebox{1}{2}};
        \end{tikzpicture}}}
    \]
\end{figure}

\begin{rmk}
    The condition that the black vertices have at least two children ensure that $\RGT(\underline{n})$ is finite for all $n\in\N$. The sequence of dimensions of $\RGT$ is Sequence~\oeis{A005264} of the OEIS \cite{Oeis}, the triangle of its graded dimensions is Sequence~\oeis{A048160}.
\end{rmk}

\begin{dfn}
    Let $S$ and $T$ be two rooted Greg trees and $i$ be a label of a white vertex $v_i$ of $S$. Let $B$ be the hypertree bellow the vertex $v_i$ in $S$, and $C=\{C_1,\dots,C_n\}$ be the set of children of $v_i$ in $S$. The \emph{insertion} of $T$ in $S$ at the vertex $i$ denoted $S\circ_i T$ is the formal sum of all possible way to graft the rooted trees $C_1,\dots,C_n$ on black or white vertices of $T$ and then grafting the result on the parent of $v_i$ in $B$.
\end{dfn}

\begin{prp}\cite{Me}
    The insertions satisfy the parallel and sequential axioms. Hence, they give a structure of operad on the species of rooted Greg trees.
\end{prp}

\begin{thm}\cite[Corollary 3.3]{Me}
    The operad $(\RGT,\{\circ_i\})$ is Koszul. Moreover, it admits the following binary quadratic presentation:
    \begin{multline*}
        \Tt\{x,y,g\}/\langle (x\circ_1 x - x\circ_2 x) - (x\circ_1 x - x\circ_2 x).(2\;3),\\ (x\circ_1 g - (g\circ_1 x).(2\;3) - g\circ_2 x) - (x\circ_1 g - (g\circ_1 x).(2\;3) - g\circ_2 x).(2\;3)\rangle ,
    \end{multline*}
    where $x$, $y$ and $g$ are operations of arity $2$, and the action of $\Sym_2$ on $\K\langle x,y,g\rangle$ is given by $x.(1\; 2)=y$ and $g.(1\;2)=g$. The isomorphism is given by $ \vcenter{\hbox{\begin{tikzpicture}[scale=0.25]     
        \node (2) at (0, 0) {};
        \node (1) at (0, -3) {};
        \draw[thick] (1)--(2);
        \draw[circle, fill=white] (1) circle [radius=24pt];
        \draw[circle, fill=white] (2) circle [radius=24pt];
        \node at (1) {\scalebox{1}{$1$}};
        \node at (2) {\scalebox{1}{$2$}};
    \end{tikzpicture}}} \mapsto x$ and $\vcenter{\hbox{\begin{tikzpicture}[scale=0.25]
        \draw[thick] (0, 0)--(-1.5, 3);
        \draw[thick] (0, 0)--(1.5, 3);
        \draw[fill=black, thick] (0, 0) circle [radius=24pt];
        \draw[fill=white, thick] (-1.5, 3) circle [radius=24pt];
        \draw[fill=white, thick] (1.5, 3) circle [radius=24pt];
        \node at (0, 0) { };
        \node at (-1.5, 3) {\scalebox{1}{$1$}};
        \node at (1.5, 3) {\scalebox{1}{$2$}};
    \end{tikzpicture}}} \mapsto g$.
\end{thm}

\section{The operad $\greg$ and the operadic twisting of $\PL$}\label{sec:greg}

Let us relate the operad $\greg$ to the operadic twisting of $\PL$. We will use the cohomological convention for the degree, hence a differential is a map a degree $1$ that square to zero. Let us describe the operadic twisting, we refer to \cite{TW} for a general introduction to the operadic twisting. Let $\g$ be a differential graded Lie algebra, a Maurer-Cartan element of $\g$ is a degree $1$ element $\alpha\in\g$ such that $d\alpha+\frac{1}{2}[\alpha,\alpha]=0$. This condition ensure that the map $d^\alpha = d+[\alpha,\cdot]$ is a differential on $\g$.

\begin{dfn}
    Let $\Pp$ be an operad. The pre-Lie algebra $(\g,\star)$ associated to $\Pp$ is the graded vector space $\g=\bigoplus_{n\in\N}\Pp(\underline{n})$ with the product $\star$ defined by:
    $$ \mu\star\nu = \sum_{i=1}^n \mu\circ_i\nu $$
    In particular, this is a Lie algebra. If the operad is differential graded, $(\g,\star)$ is a differential graded pre-Lie algebra. The Maurer-Cartan equation can be written as:
    $$ d\mu + \mu\star\mu =0 $$
    This impose that $\mu$ is an arity $1$, degree $1$ element, hence $\mu\star\mu=\mu\circ_1\mu$. Such an element is called an \emph{operadic Maurer-Cartan element}.
\end{dfn}

\begin{dfn}
    Let $\Pp$ be an operad, and $\varphi:\Lie\to\Pp$ a morphism of operads from the operad $\Lie$ to $\Pp$ and $\tilde{l}$ the image of $l$ in $\Pp$. The operadic twisting of $\Pp$ by $\varphi$ is the differential graded operad $(\Tw\Pp,d_\Tw)$ defined by as follows:
    \begin{itemize}
        \item Let $\alpha$ be a formal Maurer-Cartan element, $\alpha$ is an arity $0$, degree $1$ operation symbol.
        \item Let $\Tw\Pp=\Pp\hat{\vee}\alpha$ be the operad $\Pp$ extended by the operation symbol $\alpha$ without any relation. The symbol $\vee$ denotes the coproduct in the category of operads. We use the notation $\hat{\vee}$ since we need complete to the operad $\Pp\vee\alpha$ because of the appearance of potentially infinite sums in the general theory developed in \cite{TW}. In our case, this technicality is not relevant.
        \item The differential $d_\MC$ is defined by: $d_\MC(\alpha)=-\frac{1}{2}\tilde{l}(\alpha,\alpha)$ which is $-\frac{1}{2}(\tilde{l}\circ_1\alpha)\circ_2\alpha$ when written with the partial compositions, and for any $p\in\Pp$, $d_\MC(p)=0$. It extends to the whole operad $\Tw\Pp$ by compatibility with the composition. With this differential, any $(\Tw\Pp,d_\MC)$--algebra is a graded differential $\Pp$--algebra, with a marked Maurer-Cartan element which is the image of $\alpha$. 
        \item Let $\mu$ be an operadic Maurer-Cartan element of $(\Tw\Pp,d_\MC)$. The operadic Maurer-Cartan equation ensure that the map $d_\MC + \mu\star\cdot - \cdot\star\mu$ is a differential on $\Tw\Pp$. The differential $d_\Tw$ is defined by $d_\Tw = d_\MC + \mu\star\cdot - \cdot\star\mu$ with $\mu=\tilde{l}(\alpha,\cdot)$ which is an operadic Maurer-Cartan element.
    \end{itemize}
    We refer to \cite{TW} for the general theory of the operadic twisting.
\end{dfn}   

Let us now explicitly describe the operadic twisting of $\PL$ by the morphism $\varphi:\Lie\to\PL$ given by $\varphi(l)=x-y$, using the combinatorial description of $\PL$. Let $\alpha$ be a formal Maurer-Cartan element, $\alpha$ is an arity $0$, degree $1$ operation symbol, let us denote it by a black vertex. Then, the underlying species of $\PL\hat{\vee}\alpha$ is the species of rooted trees with black and white vertex such that white vertices are labeled, black vertices are unlabeled, and such that they have no non-trivial tree automorphism. As example, the element 
$$ ((x\circ_1 x- x\circ_2 x)\circ_2 \alpha)\circ_2\alpha $$
would be represented by the rooted tree with the white vertex as the root having two black children. However, computations show that this element is equal to its opposite, hence is zero. This species contain the species of rooted Greg trees as a subspecies, however, it is infinite dimensional in each arity. The differential $d_\MC$ is defined by $d_\MC(\alpha)=-\frac{1}{2}\tilde{l}(\alpha,\alpha)$, hence, $d_\MC(\alpha)=-\frac{1}{2}(x(\alpha,\alpha)-y(\alpha,\alpha))=-x(\alpha,\alpha)$ by the Koszul sign rule. The vector space of arity $1$ degree $1$ elements is spanned by $x\circ_1\alpha$ and $x\circ_2\alpha$, let us compute their differential:
$$ d_\MC(x\circ_1\alpha) = x\circ_1d_\MC(\alpha)=-x\circ_1((x\circ_1\alpha)\circ_2\alpha) $$
$$ d_\MC(x\circ_2\alpha) = x\circ_2d_\MC(\alpha)=-x\circ_2((x\circ_1\alpha)\circ_2\alpha) $$
Those computations can be represented using rooted trees, see Figure~\ref{fig:DiffMC}, however one need to be careful with the order in which the black vertices are ``filled'', indeed $(x\circ_1\alpha)\circ_2\alpha$ and $(x\circ_2\alpha)\circ_1\alpha$ are represented by the same rooted tree but have opposite signs, we use the convention bottom to top and left to right.

\begin{figure}
    \caption{Example of computation of $d_\MC$ on some trees.}
    \label{fig:DiffMC}
    \[
    d_\MC\;\vcenter{\hbox{\begin{tikzpicture}[scale=0.25]     
        \node (3) at (0, 0) {};
        \draw[circle, fill=black] (3) circle [radius=24pt];
    \end{tikzpicture}}}=-\vcenter{\hbox{\begin{tikzpicture}[scale=0.25]     
        \node (3) at (0, 0) {};
        \node (1) at (0, -3) {};
        \draw[thick] (1)--(3);
        \draw[circle, fill=black] (1) circle [radius=24pt];
        \draw[circle, fill=black] (3) circle [radius=24pt];
    \end{tikzpicture}}}\quad;\quad
    d_\MC\;\vcenter{\hbox{\begin{tikzpicture}[scale=0.25]     
        \node (3) at (0, 0) {};
        \node (1) at (0, -3) {};
        \draw[thick] (1)--(3);
        \draw[circle, fill=white] (3) circle [radius=24pt];
        \draw[circle, fill=black] (1) circle [radius=24pt];
        \node at (3) {\scalebox{1}{$1$}};
    \end{tikzpicture}}}=-\;\vcenter{\hbox{\begin{tikzpicture}[scale=0.25]     
        \node (3) at (0, 0) {};
        \node (1) at (0, -3) {};
        \node (2) at (0, -6) {};
        \draw[thick] (1)--(3);
        \draw[thick] (2)--(3);
        \draw[circle, fill=white] (3) circle [radius=24pt];
        \draw[circle, fill=black] (1) circle [radius=24pt];
        \draw[circle, fill=black] (2) circle [radius=24pt];
        \node at (3) {\scalebox{1}{$1$}};
    \end{tikzpicture}}}-\vcenter{\hbox{\begin{tikzpicture}[scale=0.25]     
        \node (3) at (-1.5, 0) {};
        \node (1) at (0, -3) {};
        \node (2) at (1.5, 0) {};
        \draw[thick] (1)--(3);
        \draw[thick] (2)--(1);
        \draw[circle, fill=white] (2) circle [radius=24pt];
        \draw[circle, fill=black] (1) circle [radius=24pt];
        \draw[circle, fill=black] (3) circle [radius=24pt];
        \node at (2) {\scalebox{1}{$1$}};
    \end{tikzpicture}}}\quad;\quad
    d_\MC\;\vcenter{\hbox{\begin{tikzpicture}[scale=0.25]     
        \node (1) at (0, 0) {};
        \node (3) at (0, -3) {};
        \draw[thick] (1)--(3);
        \draw[circle, fill=white] (3) circle [radius=24pt];
        \draw[circle, fill=black] (1) circle [radius=24pt];
        \node at (3) {\scalebox{1}{$1$}};
    \end{tikzpicture}}}=-\;\vcenter{\hbox{\begin{tikzpicture}[scale=0.25]     
        \node (2) at (0, 0) {};
        \node (1) at (0, -3) {};
        \node (3) at (0, -6) {};
        \draw[thick] (1)--(3);
        \draw[thick] (2)--(3);
        \draw[circle, fill=white] (3) circle [radius=24pt];
        \draw[circle, fill=black] (1) circle [radius=24pt];
        \draw[circle, fill=black] (2) circle [radius=24pt];
        \node at (3) {\scalebox{1}{$1$}};
    \end{tikzpicture}}}
    \]
\end{figure}

\begin{figure}
    \caption{Composition of the arity $1$ degree $1$ elements.}
    \label{fig:Compo11}
    \[
        (x\circ_1\alpha)\circ_1(x\circ_1\alpha) = \;\vcenter{\hbox{\begin{tikzpicture}[scale=0.25]     
            \node (3) at (0, 0) {};
            \node (1) at (0, -3) {};
            \node (2) at (0, -6) {};
            \draw[thick] (1)--(3);
            \draw[thick] (2)--(3);
            \draw[circle, fill=white] (3) circle [radius=24pt];
            \draw[circle, fill=black] (1) circle [radius=24pt];
            \draw[circle, fill=black] (2) circle [radius=24pt];
            \node at (3) {\scalebox{1}{$1$}};
        \end{tikzpicture}}}\quad;\quad (x\circ_2\alpha)\circ_1(x\circ_1\alpha)= -\;\vcenter{\hbox{\begin{tikzpicture}[scale=0.25]     
            \node (2) at (0, 0) {};
            \node (3) at (0, -3) {};
            \node (1) at (0, -6) {};
            \draw[thick] (1)--(3);
            \draw[thick] (2)--(3);
            \draw[circle, fill=white] (3) circle [radius=24pt];
            \draw[circle, fill=black] (1) circle [radius=24pt];
            \draw[circle, fill=black] (2) circle [radius=24pt];
            \node at (3) {\scalebox{1}{$1$}};
        \end{tikzpicture}}}-\vcenter{\hbox{\begin{tikzpicture}[scale=0.25]     
            \node (3) at (-1.5, 0) {};
            \node (1) at (0, -3) {};
            \node (2) at (1.5, 0) {};
            \draw[thick] (1)--(3);
            \draw[thick] (2)--(1);
            \draw[circle, fill=white] (2) circle [radius=24pt];
            \draw[circle, fill=black] (1) circle [radius=24pt];
            \draw[circle, fill=black] (3) circle [radius=24pt];
            \node at (2) {\scalebox{1}{$1$}};
        \end{tikzpicture}}}
    \]
    \[
        (x\circ_1\alpha)\circ_1(x\circ_2\alpha) = \;\vcenter{\hbox{\begin{tikzpicture}[scale=0.25]     
            \node (1) at (0, 0) {};
            \node (3) at (0, -3) {};
            \node (2) at (0, -6) {};
            \draw[thick] (1)--(3);
            \draw[thick] (2)--(3);
            \draw[circle, fill=white] (3) circle [radius=24pt];
            \draw[circle, fill=black] (1) circle [radius=24pt];
            \draw[circle, fill=black] (2) circle [radius=24pt];
            \node at (3) {\scalebox{1}{$1$}};
        \end{tikzpicture}}}\quad;\quad (x\circ_2\alpha)\circ_1(x\circ_2\alpha)= -\;\vcenter{\hbox{\begin{tikzpicture}[scale=0.25]     
            \node (2) at (0, 0) {};
            \node (1) at (0, -3) {};
            \node (3) at (0, -6) {};
            \draw[thick] (1)--(3);
            \draw[thick] (2)--(3);
            \draw[circle, fill=white] (3) circle [radius=24pt];
            \draw[circle, fill=black] (1) circle [radius=24pt];
            \draw[circle, fill=black] (2) circle [radius=24pt];
            \node at (3) {\scalebox{1}{$1$}};
        \end{tikzpicture}}}
    \]
\end{figure}

Let us find a Maurer-Cartan element. We need to compute $(x\circ_i\alpha)\circ_1(x\circ_j\alpha)$ with $i,j\in\{1,2\}$, see Figure~\ref{fig:Compo11}. Since an operadic Maurer-Cartan element is a degree one arity one element, this shows that the unique operadic Maurer-Cartan element (up to multiplication by a scalar) is $\mu=(x\circ_1\alpha)-(x\circ_2\alpha)$. This allows to describe the differential $d_\Tw$ on the rooted trees:

\begin{prp}\cite[Subsection~6.7]{PLTW}\label{prp:descrdGreg}
    Let $T$ be a rooted tree with black and white vertex, then $d_\Tw(T)$ is given by:
    \begin{enumerate}
        \item The sum of all possible ways to split a white vertex of $T$ into a white vertex retaining the label and a black vertex above it and to connect the incoming edges to one of the two vertices, taken with a minus sign.
        \item The sum of all possible ways to split a white vertex of $T$ into a white vertex retaining the label and a black vertex bellow it and to connect the incoming edges to one of the two vertices, taken with a minus sign.
        \item The sum of all possible ways to split a black vertex of $T$ into two black vertices and to connect the incoming edges to one of the two vertices, taken with a plus sign.
        \item The sum of all possible ways to graft an additional black leaf to $T$, taken with a plus sign.
        \item Grafting $T$ on top of a new black root, taken with a minus sign.
    \end{enumerate}
    Moreover, many terms cancel due to the signs. In particular, if $T$ has more than one vertex, all contributions from $4$ and $5$ get cancelled by contributions from $1$, $2$ and $3$.
\end{prp}

\begin{rmk}\label{rmk:sign}
    The signs given in the previous proposition depend on the order in which the black vertices are ``filled''. In this description, we assume that the newly created black vertex is filled first. The signs created when changing the ordering can be computed using the Koszul sign rule.
\end{rmk}

A direct computation show that:

\begin{prp}
    Let us denote $g=-d_\Tw(x)$ then we have that:
    $$g=\vcenter{\hbox{\begin{tikzpicture}[scale=0.25]     
            \node (3) at (-1.5, 0) {};
            \node (1) at (0, -3) {};
            \node (2) at (1.5, 0) {};
            \draw[thick] (1)--(3);
            \draw[thick] (2)--(1);
            \draw[circle, fill=white] (2) circle [radius=24pt];
            \draw[circle, fill=black] (1) circle [radius=24pt];
            \draw[circle, fill=white] (3) circle [radius=24pt];
            \node at (2) {\scalebox{1}{$2$}};
            \node at (3) {\scalebox{1}{$1$}};
        \end{tikzpicture}}}$$
    Moreover, $x$ and $g$ satisfies the relation:
    $$(x\circ_1 g - (g\circ_1 x).(2\;3) - g\circ_2 x) - (x\circ_1 g - (g\circ_1 x).(2\;3) - g\circ_2 x).(2\;3)$$
\end{prp}

\begin{dfn}
    Let us define the differential graded operad $\greg_{-1}$ as the operad $\greg$ such that $g$ is of degree $1$ and with the differential $d$ such that $d(x)=d(y)=-g$.
\end{dfn}

The above proposition proves that we have a morphism of differential graded operads $\greg_{-1}\to\Tw\PL$, moreover since rooted Greg trees have no non-trivial tree automorphisms, this morphism is injective.

\begin{thm}\cite[Theorem 5.1]{PLTW}
    The embedding of differential graded operads $$(\Lie,0)\to\Tw\PL$$ induces an isomorphism in the cohomology.
\end{thm}

\begin{crl}
    The embedding of differential graded operads $$(\Lie,0)\to\greg_{-1}$$ induces an isomorphism $\Lie\to H^*(\greg_{-1})$.
\end{crl}

\begin{proof}
    It is enough to prove that the complex $\Tw\PL$ splits in $\greg_{-1}\oplus \mathcal{X}$ with $\mathcal{X}$ spanned by all non-Greg trees in $\Tw\PL$. Let $T$ be a non-Greg tree, then it has a black vertex with one child or without children, from the explicit description of the differential, in each tree appearing in $dT$, this black vertex will still have one child or no children.
\end{proof}

\section{Forests of rooted hypertrees and the operad $\CPL$}\label{sec:FRHT}

Let us now give a description \emph{à la} Chapoton--Livernet of the operad $\CPL$. The operad $\CPL$ was first introduced in \cite{DefCPL}, it is defined by the following presentation:
\begin{multline*}
    \Tt\{x,y,c\}/\langle (x\circ_1 x - x\circ_2 x) - (x\circ_1 x - x\circ_2 x).(2\;3),\\ x\circ_1 c - (c\circ_1 x).(1\;3) - c\circ_2 x,\; c\circ_1 c - c\circ_2 c \rangle
\end{multline*}
where $x$, $y$ and $c$ are binary operations, and the action of $\Sym_2$ on $\K\langle x,y,c\rangle$ is given by $x.(1\; 2)=y$ and $c.(1\;2)=c$. Its Koszul dual, the operad $\CPL^!$, is defined by the following presentation:
\begin{multline*}
    \Tt\{x_*,y_*,c_*\}/\langle x_*\circ_1 x_* - x_*\circ_2 x_*,\; x_*\circ_1 x_* - (x_*\circ_2 x_*).(2\;3),\\ x_*\circ_2 c_*,\; x_*\circ_1 c_* - (c_*\circ_1 x_*).(2\;3),\; c_*\circ_1 c_* - c_*\circ_2 c_* - c_*\circ_1 c_*.(2\;3) \rangle
\end{multline*}
where $x_*$, $y_*$ and $c_*$ are binary operations, and the action of $\Sym_2$ on $\K\langle x_*,y_*,c_*\rangle$ is given by $x_*.(1\;2)=y_*$ and $c_*.(1\;2)=-c_*$. We refer to \cite[Subsection~7.2]{AlgOp} for a detailed description of the Koszul dual of a quadratic operad generated by arity $2$ elements.

In order to compute arity-wise dimensions of $\CPL^!$, let us introduce the following $\CPL^!$--algebra admitting an explicit description: 

\begin{dfn}
  Let $X$ a finite set, $\Lie(X)$ the free Lie algebra generated by $X$ and $\uCom(X)$ the free unitary commutative associative algebra generated by $X$. Let $LC(X)=\Lie(X)\otimes \uCom(X)$. For $a_1\otimes a_2$ and $b_1\otimes b_2$ in $LC(X)$, let us define two binary operations $[\cdot, \cdot]$ and $\cdot. \cdot$ by:
  \begin{itemize}
    \item $(a_1\otimes a_2).(b_1\otimes b_2)=a_1\otimes( a_2.b_1. b_2)$ if $b_1\in \Vect(X)$;
    \item $(a_1\otimes a_2).(b_1\otimes b_2)=0$ if $b_1\notin\Vect(X)$;
    \item $[(a_1\otimes a_2), (b_1\otimes b_2)]=[a_1, b_1]\otimes( a_2. b_2)$.
  \end{itemize}
  One can check that $LC(X)$ is a $(\CPL)^!$--algebra with $[\cdot, \cdot]$ the image of $c_*$ and $\cdot. \cdot$ the image of $x_*$. Moreover, it is generated as a $(\CPL)^!$--algebra by the elements of the form $a\otimes1$ with $a\in X$ and $1$ the unit of $\uCom(X)$.
\end{dfn}

\begin{dfn}
    Let $u_n$ be the sequence of logarithmic numbers, see Sequence~\oeis{A002104} in the OEIS \cite{Oeis}. This sequence is defined by:
    $$\sum_{n\geq1}\frac{u_n}{n!}t^n=-\log(1-t)\exp(t)$$
\end{dfn}

\begin{prp}
    We have that $u_n=\dim(\Mult(LC(\{a_1,\dots,a_n\})))$ with $\Mult$ the multilinear part, in particular $u_4=24$.
\end{prp}

\begin{proof}
    Since $LC(\{a_1,\dots,a_n\})=\Lie(\{a_1,\dots,a_n\})\otimes\uCom(\{a_1,\dots,a_n\})$, we have that:
    $$\Mult(LC(\{a_1,\dots,a_n\}))=\bigoplus_{I\sqcup J=\{a_1,\dots,a_n\}}\Mult(\Lie(I))\otimes\Mult(\uCom(J))$$
    Hence, $n\to\Mult(LC(\{a_1,\dots,a_n\}))$ give raise to a species which is the Cauchy product of $\Lie$ and $\uCom$. Hence, its exponential generating series is $-\log(1-t)\exp(t)$.
\end{proof}

The operad $\CPL^!$ admits a terminating quadratic rewriting system displayed in the appendix Supplementary Figure~1, it has $19$ rules. This rewriting system is obtained using the ``quantum$(x_*,y_*\; c_*)$ permutation degree-lexicographic order'' see \cite{GB,Quant} for further details on shuffle operads and monomial orders. Moreover, this rewriting system has the following property:

\begin{prp}
    The sequence of numbers of normal form of the rewriting system displayed in Supplementary Figure~1 is the sequence of logarithmic numbers $u_n$.
\end{prp}

\begin{proof}
    Since $\CPL^!$ is graded by the number of occurrences of $c_*$, it is clear that the suboperad generated by $c_*$ is the $\Lie$ operad. The rewriting system displayed in Supplementary Figure~1 restricted to $c_*$ is the rewriting system associated to the ``permutation degree-lexicographic order'', which is known to be a Gröbner basis for the $\Lie$ operad, see \cite[Example~5.6.1.1]{GB}. In particular, normal forms of $\Lie(n)$ are in bijection with a basis of $\Lie(\{a_1,\dots,a_n\})$.\\
    An analogous observation shows that the rewriting system displayed in Supplementary Figure~1 restricted to $x_*$ and $y_*$ is a Gröbner basis of the operad $\PL^!$, also known as $\mathrm{Perm}$, see \cite{Perm}.\\
    Moreover, a normal form is given by a pair $(a,b)$ with $a$ a normal form of $\Lie$ and $b$ a normal form of $\mathrm{Perm}$, with $a$ composed in the non-symmetric input of $b$. This allows us to build a bijection between normal forms of $\CPL^!(n)$ and $\Mult(LC(\{a_1,\dots,a_n\}))$. Hence, the number of normal forms of $\CPL^!(n)$ is the number of multilinear elements of $LC(\{a_1,\dots,a_n\})$. This concludes the proof.
\end{proof}

\begin{thm}
    The operad $\CPL^!$ is Koszul. Moreover, its Hilbert series is given by:
    $$ f_{\CPL^!}(t) = -\log(1-t)\exp(t) $$
\end{thm}

\begin{proof}
    The inequality $\dim(\CPL^!(4))\geq 24$ and the fact that the rewriting system admit $24$ normal forms in arity $4$ ensures that it is a Gröbner basis. Hence, $\CPL^!$ is Koszul.
\end{proof}

\begin{crl}\label{thm:CPLK}
    The operad $\CPL$ is Koszul. Moreover, its Hilbert series is given by:
    $$ f_{\CPL}(t) = \rev(\log(1+t)\exp(-t)) $$
    with $\rev$ the compositional inverse in $t$ of a series.
\end{crl}

With this result, we know the dimensions of the operad $\CPL$, indeed it is Sequence~\oeis{A052888} of the OEIS \cite{Oeis}. Let us now give a combinatorial description of the operad $\CPL$. To do so, let us introduce the hypergraphs and hypertrees from \cite{HT}.

\begin{dfn}
    An \emph{hypergraph} is a pair $(\V,\E)$ where $\V$ is a finite set and $\E$ is a subset of $\mathcal{P}_{\geq2}(\V)$, the subsets of $\V$ of cardinality at least $2$. The elements of $\V$ are called \emph{vertices} and the elements of $\E$ are called \emph{edges}. The \emph{weight} of an edge is its cardinality minus $2$. The edges of positive weight are the \emph{hyperedges}, the edges of weight zero are the \emph{simple edge}.
\end{dfn}

\begin{dfn}
    A \emph{path} of length $n$ is a pair $((v_0,\dots, v_n),(e_1,\dots,e_n))$ where $(v_0,\dots, v_n)$ is a sequence of vertices and $(e_1,\dots,e_n)$ is a sequence of edges such that for all $i\in\{1,\dots,n\}$, $v_{i-1}\in e_i$ and $v_i\in e_i$. A \emph{cycle} is a path such that $v_0=v_n$, and all the edges $e_i$ are distinct. 
\end{dfn}

\begin{dfn}
    A \emph{rooted hypertree} $\tau$ is a hypergraph $(\V,\E)$ without cycles, and with a distinguished vertex called the \emph{root}. The \emph{weight} of a rooted hypertree is the sum of the weights of its hyperedges. A \emph{labeling} on $\tau$ is a bijective map $l:\V\to L$ with $L$ a set of labels. The species of rooted hypertrees is the species $\RHT$ such that $\RHT(L)$ is the set of rooted hypertrees labeled by the finite set $L$. The species $\RHT$ is graded by the weight, in particular, $\RHT_0=\RT$. In the following, rooted hypertrees will be labeled. Rooted hypertrees will be represented with the root at the bottom. 
\end{dfn}

\begin{rmk}
    In particular, one may remark that with this definition two vertices of a hypertree share at most one edge.
\end{rmk}

\begin{dfn}
    A \emph{forest of rooted hypertrees} is a non-empty finite set of rooted hypertrees labeled over disjoint sets of labels. The \emph{weight} of a forest of rooted hypertrees is the number of rooted hypertrees in the forest minus $1$ plus the sum of the weight of the rooted hypertrees in the forest. The species of forests of rooted hypertrees is the species $\FRHT$ such that $\FRHT(L)$ is the set of forests of rooted hypertrees labeled by the finite set $L$. The species $\FRHT$ is graded by the weight, in particular, $\FRHT_0=\RT$. The set $\FRHT(\underline{2})$ is depicted in Figure \ref{fig:FRHT2}.
\end{dfn}

\begin{figure}[ht]
    \caption{The set $\FRHT(\underline{2})$}
    \label{fig:FRHT2}
    \[
        \vcenter{\hbox{\begin{tikzpicture}[scale=0.25]
            \draw[thick] (0, 0)--(0, 3);
            \draw[fill=white, thick] (0, 0) circle [radius=24pt];
            \draw[fill=white, thick] (0, 3) circle [radius=24pt];
            \node at (0, 0) {\scalebox{1}{1}};
            \node at (0, 3) {\scalebox{1}{2}};
        \end{tikzpicture}}}
        \quad,\quad
        \vcenter{\hbox{\begin{tikzpicture}[scale=0.25]
            \draw[thick] (0, 0)--(0, 3);
            \draw[fill=white, thick] (0, 0) circle [radius=24pt];
            \draw[fill=white, thick] (0, 3) circle [radius=24pt];
            \node at (0, 0) {\scalebox{1}{2}};
            \node at (0, 3) {\scalebox{1}{1}};
            \end{tikzpicture}}}
        \quad,\quad
        \vcenter{\hbox{\begin{tikzpicture}[scale=0.25]
            \draw[fill=white, thick] (0, 0) circle [radius=24pt];
            \draw[fill=white, thick] (3, 0) circle [radius=24pt];
            \node at (0, 0) {\scalebox{1}{1}};
            \node at (3, 0) {\scalebox{1}{2}};
        \end{tikzpicture}}}
    \]
\end{figure}

Let $u$ be a formal variable encoding the weight grading. This means that $f_\FRHT(t,u)=\sum a_{i,j}\frac{t^iu^j}{i!}$ such that $a_{i,j}$ is the number of forests of rooted hypertrees with $i$ vertices and of weight $j$. 

\begin{prp}
    We have:
    $$ f_\FRHT(t,u)=\rev\left(\frac{\ln(1+ut)}{u}\exp(-t)\right) $$
    The triangle of number $\dim(\FRHT_k(n))$ is given by Sequence~\oeis{A364709}. 
\end{prp}

\begin{proof}
    The graded species $\FRHT$ is given by the following formula:
    $$ \FRHT=\frac{1}{u}E_{\geq 1}(u.\RHT) $$
    with the notation that $E_{\geq 1}$ is the non-empty set species. Hence we have:
    $$ f_\FRHT=\frac{1}{u}(\exp(u.f_\RHT)-1) $$
    Hence:
    $$ f_\RHT=\frac{\ln(1+u.f_\FRHT)}{u} $$
    Moreover, the graded species $\RHT$ is given by the following formula:
    $$ \RHT=X.E(\FRHT) $$
    with the notation that $X$ is the singleton species and $E$ the set species. Hence we have:
    $$ \frac{\ln(1+u.f_\FRHT)}{u}\exp(-f_\FRHT)=t $$
    which concludes the proof.
\end{proof}

\begin{dfn}
    Let $S$ and $T$ be two forests of rooted hypertrees and $i$ be a label of a vertex $v_i$ of $S$. Let $B$ be the tree bellow the vertex $v_i$ in $S$, and $C=\{C_1,\dots,C_n\}$ be the set of forests of children of $v_i$ in $S$ such that each rooted hypertrees that are grafted at $v_i$ by the same edge are in the same forest. The \emph{insertion} of $T$ in $S$ at the vertex $i$ denoted $S\circ_i T$ is the formal sum of all possible way to graft the set of forests of rooted hypertrees $C_1,\dots,C_n$ on vertices of $T$ such that each rooted hypertrees of $C_j$ are grafted at $T$ by the same edge, and then grafting the result on the parent of $v_i$ in $B$. If $T$ is a forest this creates a unique hyperedge that connects all its rooted hypertrees to the parent of $v_i$. 
\end{dfn}

Let us compute the following examples:
\[
    \vcenter{\hbox{\begin{tikzpicture}[scale=0.25]     
        \node (3) at (0, 0) {};
        \node (1) at (0, -3) {};
        \draw[thick] (1)--(3);
        \draw[circle, fill=white] (1) circle [radius=24pt];
        \draw[circle, fill=white] (3) circle [radius=24pt];
        \node at (1) {\scalebox{1}{$1$}};
        \node at (3) {\scalebox{1}{$3$}};
    \end{tikzpicture}}} \; \circ_1 \; \vcenter{\hbox{\begin{tikzpicture}[scale=0.25]     
        \node (1) at (0, 0) {};
        \node (2) at (3, 0) {};
        \draw[circle, fill=white] (1) circle [radius=24pt];
        \draw[circle, fill=white] (2) circle [radius=24pt];
        \node at (1) {\scalebox{1}{$1$}};
        \node at (2) {\scalebox{1}{$2$}};
    \end{tikzpicture}}} \; = \; \vcenter{\hbox{\begin{tikzpicture}[scale=0.25]     
        \node (3) at (0, 0) {};
        \node (2) at (0, -3) {};
        \node (1) at (3, -3) {};
        \draw[thick] (3)--(2);
        \draw[circle, fill=white] (1) circle [radius=24pt];
        \draw[circle, fill=white] (2) circle [radius=24pt];
        \draw[circle, fill=white] (3) circle [radius=24pt];
        \node at (1) {\scalebox{1}{$1$}};
        \node at (2) {\scalebox{1}{$2$}};
        \node at (3) {\scalebox{1}{$3$}};
    \end{tikzpicture}}} \; + \; \vcenter{\hbox{\begin{tikzpicture}[scale=0.25]     
        \node (2) at (3, -3) {};
        \node (1) at (0, -3) {};
        \node (3) at (0, 0) {};
        \draw[thick] (1)--(3);
        \draw[circle, fill=white] (1) circle [radius=24pt];
        \draw[circle, fill=white] (2) circle [radius=24pt];
        \draw[circle, fill=white] (3) circle [radius=24pt];
        \node at (1) {\scalebox{1}{$1$}};
        \node at (2) {\scalebox{1}{$2$}};
        \node at (3) {\scalebox{1}{$3$}};
    \end{tikzpicture}}}
\]
and:
\[
    \vcenter{\hbox{\begin{tikzpicture}[scale=0.25]     
        \node (3) at (0, 0) {};
        \node (1) at (0, -3) {};
        \draw[thick] (1)--(3);
        \draw[circle, fill=white] (1) circle [radius=24pt];
        \draw[circle, fill=white] (3) circle [radius=24pt];
        \node at (1) {\scalebox{1}{$1$}};
        \node at (3) {\scalebox{1}{$2$}};
    \end{tikzpicture}}} \; \circ_2 \; \vcenter{\hbox{\begin{tikzpicture}[scale=0.25]     
        \node (1) at (0, 0) {};
        \node (2) at (3, 0) {};
        \draw[circle, fill=white] (1) circle [radius=24pt];
        \draw[circle, fill=white] (2) circle [radius=24pt];
        \node at (1) {\scalebox{1}{$2$}};
        \node at (2) {\scalebox{1}{$3$}};
    \end{tikzpicture}}} \; = \; \vcenter{\hbox{\begin{tikzpicture}[scale=0.25]     
        \node (2) at (-1.5, 0) {};
        \node (1) at (0, -3) {};
        \node (3) at (1.5, 0) {};
        \fill[gray] (1.north)--(2.south)--(3.south)--cycle;
        \draw[thick] (1.north)--(2.south);
        \draw[thick] (1.north)--(3.south);
        \draw[thick] (2.south)--(3.south);
        \draw[circle, fill=white] (1) circle [radius=24pt];
        \draw[circle, fill=white] (2) circle [radius=24pt];
        \draw[circle, fill=white] (3) circle [radius=24pt];
        \node at (1) {\scalebox{1}{$1$}};
        \node at (2) {\scalebox{1}{$2$}};
        \node at (3) {\scalebox{1}{$3$}};
    \end{tikzpicture}}} 
\]
As one can remark, a hyperedge is created in the second example.

\begin{prp}
    The insertions satisfy the parallel and sequential axioms. Hence, they give a structure of operad on the species of forest of rooted hypertrees. 
\end{prp}

\begin{proof}
    It is clear that the parallel axiom is verified since we are inserting forests of rooted hypertrees in different vertices. The proof of the sequential axiom is the computation shown in Figure~\ref{fig:FRHTseq} with the convention that double edges means that all the rooted hypertrees of the forest are grafted at the same vertex via the same edge. Indices are omitted for readability.
\end{proof}

\begin{figure}[ht]
    \caption{The sequential axiom for $\FRHT$}\label{fig:FRHTseq}
    \begin{subfigure}{.45\textwidth}\subcaption{Notation for $T$}
        \centering
        \[
            \vcenter{\hbox{\begin{tikzpicture}[scale=0.43]     
                \node (S) at (0, 0) {};
                \node[isosceles triangle, isosceles triangle apex angle=60, draw, rotate=270, fill=white!50, minimum size=24pt] (ST) at (S) {};
                \draw[circle, fill=black] (ST.east) circle [radius=2pt];
                \node at (S) {\scalebox{1}{$T$}};
            \end{tikzpicture}}} =
            \vcenter{\hbox{\begin{tikzpicture}[scale=0.43]     
                \node (R) at (0, 0) {};
                \node (v) at (0, 3) {};
                \node (1) at (-3, 6) {};
                \node (dots) at (0, 6) {};
                \node (2) at (3, 6) {};
                \node[isosceles triangle, isosceles triangle apex angle=60, draw, rotate=270, fill=white!50, minimum size=24pt] (RT) at (R) {};
                \node[isosceles triangle, isosceles triangle apex angle=60, draw, rotate=270, fill=white!50, minimum size=24pt] (1T) at (1) {};
                \node[isosceles triangle, isosceles triangle apex angle=60, draw, rotate=270, fill=white!50, minimum size=24pt] (2T) at (2) {};
                \draw[thick] (RT.west)--(v);
                \draw[thick,double] (1T.east)--(v);
                \draw[thick,double] (2T.east)--(v);
                \node at (dots) {$\cdots$};
                \draw[circle, fill=white] (v) circle [radius=24pt];
                \draw[circle, fill=black] (1T.east) circle [radius=2pt];
                \draw[circle, fill=black] (2T.east) circle [radius=2pt];
                \draw[circle, fill=black] (RT.east) circle [radius=2pt];
                \node at (1) {\scalebox{1}{$C_1$}};
                \node at (2) {\scalebox{1}{$C_k$}};
                \node at (R) {\scalebox{1}{$A$}};
                \node at (v) {\scalebox{1}{$v$}};
            \end{tikzpicture}}}
        \]
    \end{subfigure}
    \begin{subfigure}{.45\textwidth}\subcaption{Notation for $S$}
        \centering
        \[
            \vcenter{\hbox{\begin{tikzpicture}[scale=0.43]     
                \node (S) at (0, 0) {};
                \node[isosceles triangle, isosceles triangle apex angle=60, draw, rotate=270, fill=white!50, minimum size=24pt] (ST) at (S) {};
                \draw[circle, fill=black] (ST.east) circle [radius=2pt];
                \node at (S) {\scalebox{1}{$S$}};
            \end{tikzpicture}}} =
            \vcenter{\hbox{\begin{tikzpicture}[scale=0.43]     
                \node (R) at (0, 0) {};
                \node (v) at (0, 3) {};
                \node (1) at (-3, 6) {};
                \node (dots) at (0, 6) {};
                \node (2) at (3, 6) {};
                \node[isosceles triangle, isosceles triangle apex angle=60, draw, rotate=270, fill=white!50, minimum size=24pt] (RT) at (R) {};
                \node[isosceles triangle, isosceles triangle apex angle=60, draw, rotate=270, fill=white!50, minimum size=24pt] (1T) at (1) {};
                \node[isosceles triangle, isosceles triangle apex angle=60, draw, rotate=270, fill=white!50, minimum size=24pt] (2T) at (2) {};
                \draw[thick] (RT.west)--(v);
                \draw[thick,double] (1T.east)--(v);
                \draw[thick,double] (2T.east)--(v);
                \node at (dots) {$\cdots$};
                \draw[circle, fill=white] (v) circle [radius=24pt];
                \draw[circle, fill=black] (1T.east) circle [radius=2pt];
                \draw[circle, fill=black] (2T.east) circle [radius=2pt];
                \draw[circle, fill=black] (RT.east) circle [radius=2pt];
                \node at (1) {\scalebox{1}{$D_1$}};
                \node at (2) {\scalebox{1}{$D_\ell$}};
                \node at (R) {\scalebox{1}{$B$}};
                \node at (v) {\scalebox{1}{$w$}};
            \end{tikzpicture}}}
        \]
    \end{subfigure}
    \begin{subfigure}{\textwidth}\subcaption{Computation of $T\circ_v S\circ_w R$}
        \centering
        \[
            \vcenter{\hbox{\begin{tikzpicture}[scale=0.43]     
                \node (C1) at (-13.5, 12) {};
                \node (dot1) at (-11.5, 12) {};
                \node (C2) at (-9.5, 12) {};
                \node (C3) at (-6.5, 12) {};
                \node (dot2) at (-4.5, 12) {};
                \node (C4) at (-2.5, 12) {};
                \node (D1) at (-11.5, 9) {};
                \node (dot3) at (-8, 9) {};
                \node (D2) at (-4.5, 9) {};
                \node (C5) at (-1.5, 9) {};
                \node (dot4) at (0.5, 9) {};
                \node (C6) at (2.5, 9) {};
                \node (R) at (-3, 6) {};
                \node (C7) at (1, 6) {};
                \node (dot5) at (3, 6) {};
                \node (C8) at (5, 6) {};
                \node (B) at (0, 3) {};
                \node (A) at (0, 0) {};
                \node[isosceles triangle, isosceles triangle apex angle=60, draw, rotate=270, fill=white!50, minimum size=24pt] (C1T) at (C1) {};
                \node[isosceles triangle, isosceles triangle apex angle=60, draw, rotate=270, fill=white!50, minimum size=24pt] (C2T) at (C2) {};
                \node[isosceles triangle, isosceles triangle apex angle=60, draw, rotate=270, fill=white!50, minimum size=24pt] (C3T) at (C3) {};
                \node[isosceles triangle, isosceles triangle apex angle=60, draw, rotate=270, fill=white!50, minimum size=24pt] (C4T) at (C4) {};
                \draw[thick,double] (C1T.east)--(D1);
                \draw[thick,double] (C2T.east)--(D1);
                \draw[thick,double] (C3T.east)--(D2);
                \draw[thick,double] (C4T.east)--(D2);
                \node[isosceles triangle, isosceles triangle apex angle=60, draw, rotate=270, fill=white!50, minimum size=24pt] (C5T) at (C5) {};
                \node[isosceles triangle, isosceles triangle apex angle=60, draw, rotate=270, fill=white!50, minimum size=24pt] (C6T) at (C6) {};
                \node[isosceles triangle, isosceles triangle apex angle=60, draw, rotate=270, fill=white!50, minimum size=24pt] (C7T) at (C7) {};
                \node[isosceles triangle, isosceles triangle apex angle=60, draw, rotate=270, fill=white!50, minimum size=24pt] (C8T) at (C8) {};
                \node[isosceles triangle, isosceles triangle apex angle=60, draw, rotate=270, fill=white!50, minimum size=24pt] (D1T) at (D1) {};
                \node[isosceles triangle, isosceles triangle apex angle=60, draw, rotate=270, fill=white!50, minimum size=24pt] (D2T) at (D2) {};
                \draw[thick,double] (C5T.east)--(R);
                \draw[thick,double] (C6T.east)--(R);
                \draw[thick,double] (D1T.east)--(R);
                \draw[thick,double] (D2T.east)--(R);
                \draw[thick,double] (C7T.east)--(B);
                \draw[thick,double] (C8T.east)--(B);
                \node[isosceles triangle, isosceles triangle apex angle=60, draw, rotate=270, fill=white!50, minimum size=24pt] (RT) at (R) {};
                \draw[thick,double] (RT.east)--(B);
                \node[isosceles triangle, isosceles triangle apex angle=60, draw, rotate=270, fill=white!50, minimum size=24pt] (BT) at (B) {};
                \draw[thick,double] (BT.east)--(A);
                \node[isosceles triangle, isosceles triangle apex angle=60, draw, rotate=270, fill=white!50, minimum size=24pt] (AT) at (A) {};
                \node at (dot1) {$\cdots$};
                \node at (dot2) {$\cdots$};
                \node at (dot3) {$\cdots$};
                \node at (dot4) {$\cdots$};
                \node at (dot5) {$\cdots$};
                \draw[circle, fill=black] (C1T.east) circle [radius=2pt];
                \draw[circle, fill=black] (C2T.east) circle [radius=2pt];
                \draw[circle, fill=black] (C3T.east) circle [radius=2pt];
                \draw[circle, fill=black] (C4T.east) circle [radius=2pt];
                \draw[circle, fill=black] (C5T.east) circle [radius=2pt];
                \draw[circle, fill=black] (C6T.east) circle [radius=2pt];
                \draw[circle, fill=black] (C7T.east) circle [radius=2pt];
                \draw[circle, fill=black] (C8T.east) circle [radius=2pt];
                \draw[circle, fill=black] (D1T.east) circle [radius=2pt];
                \draw[circle, fill=black] (D2T.east) circle [radius=2pt];
                \draw[circle, fill=black] (RT.east) circle [radius=2pt];
                \draw[circle, fill=black] (BT.east) circle [radius=2pt];
                \draw[circle, fill=black] (AT.east) circle [radius=2pt];
                \node at (C1) {\scalebox{1}{$C$}};
                \node at (C2) {\scalebox{1}{$C$}};
                \node at (C3) {\scalebox{1}{$C$}};
                \node at (C4) {\scalebox{1}{$C$}};
                \node at (C5) {\scalebox{1}{$C$}};
                \node at (C6) {\scalebox{1}{$C$}};
                \node at (C7) {\scalebox{1}{$C$}};
                \node at (C8) {\scalebox{1}{$C$}};
                \node at (D1) {\scalebox{1}{$D$}};
                \node at (D2) {\scalebox{1}{$D$}};
                \node at (R) {\scalebox{1}{$R$}};
                \node at (B) {\scalebox{1}{$B$}};
                \node at (A) {\scalebox{1}{$A$}};
            \end{tikzpicture}}}
        \]
    \end{subfigure}
\end{figure}

Let us denote:
\[
x_n=\vcenter{\hbox{\begin{tikzpicture}[scale=0.25]     
    \node (R) at (3,-3) {};
    \node (1) at (0, 0) {};
    \node (dots) at (3, 0) {};
    \node (2) at (6, 0) {};
    \draw[thick] (R) -- (1);
    \draw[thick] (R) -- (2);
    \node at (dots) {$\cdots$};
    \draw[circle, fill=white] (R) circle [radius=24pt];
    \draw[circle, fill=white] (1) circle [radius=24pt];
    \draw[circle, fill=white] (2) circle [radius=24pt];
    \node at (R) {\scalebox{1}{$1$}};
    \node at (1) {\scalebox{1}{$2$}};
    \node at (2) {\scalebox{1}{$n$}};
\end{tikzpicture}}}  
\quad ; \quad
c_n=\vcenter{\hbox{\begin{tikzpicture}[scale=0.25]     
    \node (1) at (0, 0) {};
    \node (dots) at (3, 0) {};
    \node (2) at (6, 0) {};
    \node at (dots) {$\cdots$};
    \draw[circle, fill=white] (1) circle [radius=24pt];
    \draw[circle, fill=white] (2) circle [radius=24pt];
    \node at (1) {\scalebox{1}{$1$}};
    \node at (2) {\scalebox{1}{$n$}};
\end{tikzpicture}}}
\]
As in the operad $\PL$, the elements $x_n$ are the symmetric braces, see \cite{Brace}.

\begin{prp}
    The operad $\FRHT$ is generated by arity $2$ elements. It means that $\FRHT$ is generated by $x_2$ and $c_2$.
\end{prp}

\begin{proof}
    Let $\Pp$ be the suboperad of $\FRHT$ generated by $\FRHT(2)$. Let us prove inductively that $\Pp=\FRHT$:
    \begin{itemize}
        \item Initial case: $\Pp(2)=\FRHT(2)$ by definition.
        \item Induction step: If $T=x_n$, then $T$ is a rooted tree, and thus in the suboperad generated by $x_2$ since $\PL$ is generated by arity $2$ elements. If $T=c_n$, then $T=(\dots(c_2\circ_1c_2)\circ_1\dots)\circ_1c_2$. Else, $T$ can be obtained by inductively composing copies of $x_i$ and of $c_j$ at the leaves.
    \end{itemize}
\end{proof}

\begin{thm}
    The operad $(\FRHT,\{\circ_i\})$ is isomorphic to the operad $\CPL$. Moreover, the morphism is given by $ x_2\mapsto x$ and $c_2 \mapsto c$.
\end{thm}

\begin{proof}
    The example of computation show that $\FRHT$ satisfies the relations of $\CPL$. Since, $\FRHT$ is generated by arity $2$ elements, we have a surjective morphism $\CPL\to\FRHT$. The equality of the Hilbert series show that this morphism is bijective.
\end{proof}

\section{Forests of rooted Greg hypertrees}
Now that we have a combinatorial description of the operad $\CPL$, we want an analogue of the operad $\greg$ in this context. Let us define the $\CG$ operad by the following presentation:
\begin{multline*}
    \Tt\{x,y,c,g\}/\langle (x\circ_1 x - x\circ_2 x) - (x\circ_1 x - x\circ_2 x).(2\;3),\\ x\circ_1 c - (c\circ_1 x).(1\;3) - c\circ_2 x,\; c\circ_1 c - c\circ_2 c,\\ x\circ_1 g - (g\circ_1 x).(2\;3) - g\circ_2 x \rangle
\end{multline*}
with $x$, $y$, $c$ and $g$ binary operations, and the action of $\Sym_2$ on $\K\langle x,y,c,g\rangle$ is given by $x.(1\; 2)=y$, $c.(1\;2)=c$ and $g.(1\;2)=g$. Its Koszul dual, the operad $\CG^!$ is defined by the following presentation:
\begin{multline*}
    \Tt\{x_*,y_*,c_*,g_*\}/\langle x_*\circ_1 x_* - x_*\circ_2 x_*,\; x_*\circ_1 x_* - (x_*\circ_1 x_*).(2\;3),\\ x_*\circ_2c_*,\; c_*\circ_2 x_*-x_*\circ_1 c_* ,\; c_*\circ_1 c_* - (c_*\circ_1c_*).(2\;3) - c_*\circ_2 c_*,\\ x_*\circ_1 g_* - g_*\circ_2 x_*,\; x_*\circ_2g_*,\; c_*\circ_1 g_*,\; g_*\circ_1 c_*,\;g_*\circ_1 g_* \rangle
\end{multline*}

\begin{rmk}
    Let us denote $\vee$ the coproduct of operads, and for $\Pp$ an operad, $\vee_\Pp$ the fibered coproduct of operads over $\Pp$. One may remark that $\CG=\CPL\vee_{\PL}\greg$. This is not enough to show that $\CG$ is Koszul, however as we know a description of the free $\CPL^!$--algebras and $\greg^!$--algebras see \cite{Me}, we can guess a description of the free $\CG^!$--algebras, and show that $\CG^!$ is Koszul.
\end{rmk}

\begin{dfn}
    Let $X$ be a finite alphabet. Let $\mathbf{Ar}(X)$ be the span of finite words on $X$ with the following extra decoration: there is an arrow from one letter to another. Let $Ar(X)$ be the quotient of $\mathbf{Ar}(X)$ by the following relations: letters commute with each other (the arrow follows the letters), reverting the arrow change the sign and 
    $$\overset{\curvearrowright}{ab}cv=\overset{\curvearrowright}{cb}av+\overset{\curvearrowright}{ac}bv$$
    for any $a,b,c\in X$ and $v$ a finite word. Because the letters commute, we can write the elements of $Ar(X)$ with the arrow going from the first letter to the second one.\\
    Let $LCA(X)=LC(X)\oplus Ar(X)$ and let us define $\cdot.\cdot$, $[\cdot,\cdot]$ and $\{\cdot,\cdot\}$ on $LCA(X)$ by:
    \begin{itemize}
        \item $\{a\otimes v,b\otimes w\}=\overset{\curvearrowright}{ab}vw$ for $a,b\in X$,
        \item $(\overset{\curvearrowright}{ab}v).(c\otimes w)=\overset{\curvearrowright}{ab}vcw$ for $c\in X$,
        \item and $\cdot.\cdot$ and $[\cdot,\cdot]$ are the same as in $LC(X)$.
    \end{itemize}
    All other cases give $0$.
\end{dfn}

\begin{prp}
    The algebra $(LCA(X),\cdot.\cdot,[\cdot,\cdot],\{\cdot,\cdot\})$ is a $(\CG)^!$--algebra generated by $X$.
\end{prp}

\begin{proof}
    Direct computations show that $LCA(X)$ is a $(\CG)^!$--algebra. Moreover, it is generated by $X$ since $LC(X)$ is generated by $X$ and $\overset{\curvearrowright}{ab}v=\{a\otimes v,b\otimes \varepsilon\}$.
\end{proof}

\begin{prp}
    We have $\dim(\Mult(LCA(\{a,b,c,d\})))=27$.
\end{prp}

\begin{proof}
    We have:
    $$\Mult(LCA(\{a,b,c,d\}))=\Mult(LC(\{a,b,c,d\}))\oplus\Mult(Ar(\{a,b,c,d\}))$$
    We already know the dimension of $\Mult(LC(\{a,b,c,d\}))$, and it is not difficult to check that $\Mult(Ar(\{a,b,c,d\}))$ is of dimension $3$. Hence, $\Mult(LCA(\{a,b,c,d\}))$ is of dimension $27$.
\end{proof}

The rewriting system of $\CG^!$ is displayed in the appendix Supplementary Figures~1~and~2 with the rules not involving $g_*$ in the first one and the ones involving $g_*$ in the second one, it has $38$ rules. This rewriting system is obtained using the ``quantum$(x_*,y_*\; c_*\; g_*)$ permutation degree-lexicographic order''; see \cite{GB,Quant} for further details on shuffle operads and monomial orders. Moreover, this rewriting system has the following property:

\begin{prp}
    The exponential generating function of the number of normal forms of the rewriting system displayed in Supplementary Figures~1~and~2 is given by:
    $$f=-\ln(1-t)\exp(t)+t\exp(t)-exp(t)+1$$
    In particular it has $27$ normal forms in arity $4$.
\end{prp}

\begin{proof}
    One may remark that $c_*$ and $g_*$ cannot appear at the same time in a normal form. Hence, either $g_*$ appears or not. If $g_*$ does not appear, then we have a normal form of $\CPL^!$. If $g_*$ appears, then we have a left comb with only $g_*$ and $x_*$ appearing, and only one occurrence of $g_*$ at the top of the left comb. Hence, a normal form with $g_*$ appearing is entirely determined by the label of the second leave of $g_*$, hence we have $n-1$ such normal form in arity $n$. Computation of the exponential generating series show that it is:
    $$-\ln(1-t)\exp(t)+t\exp(t)-exp(t)+1$$
\end{proof}

\begin{thm}
    The operad $\CG^!$ is Koszul. 
\end{thm}

\begin{proof}
    The inequality $\dim(\CG^!(4))\geq 27$ and the fact that the rewriting system admit $27$ normal forms in arity $4$ ensure that it is a Gröbner basis. Hence, $\CG^!$ is Koszul.
\end{proof}

\begin{crl}
    The operad $\CG$ is Koszul. Moreover, its Hilbert series is given by:
    $$ f_{\CG}(t) = \rev(\ln(1+t)\exp(-t)+t\exp(-t)+\exp(-t)-1) $$
    with $\rev$ the compositional inverse in $t$ of a series.
\end{crl}

Now that we know that the arity-wise dimensions of the operad $\CG$, we can describe the underlying species. Let us define a species that allows us to combine the Greg tree and the hypertrees. 

\begin{dfn}
    A \emph{forest of rooted Greg hypertrees} is a forest of rooted hypertrees with two kinds of vertices, the white and the black vertices such that the white vertices are labeled, and the black vertices are unlabeled and have at least two incoming edges. The species $\FRGHT$ of forests of rooted Greg hypertrees is the species such that $\FRGHT(L)$ is the set of forests of rooted Greg hypertrees labeled by the finite set $L$. The species $\FRGHT$ is bi-graded by the weight of the forests of rooted hypertrees and by the number of black vertices, to avoid confusion, let us call then respectively the \emph{hypertree weight} and the \emph{Greg weight}. In particular, $\FRGHT_{0,0}=\RT$, $\FRGHT_{0,k}=\RGT_k$ and $\FRGHT_{k,0}=\FRHT_k$. 
\end{dfn}

Let $u$ be a formal variable encoding the hypertree weight grading, and $v$ be a formal variable encoding the Greg weight grading. This means that $f_\FRGHT(t,u,v)=\sum a_{i,j,k}\frac{t^iu^jv^k}{i!}$ where $a_{i,j,k}$ is the number of forests of rooted Greg hypertrees with $i$ white vertices, of hypertree weight $j$ and of Greg weight $k$. 

\begin{prp}
    We have:
    $$ f_\FRGHT(t,u,v)=\rev\left(\left(\frac{\ln(1+ut)}{u}-v(\exp(t)-t-1) \right) \exp(-t) \right) $$
    In particular the sequence of forests of rooted Greg hypertrees with $n$ white vertices is given by Sequence~\oeis{A364816}.
\end{prp}

\begin{proof}
    Let us inspect the species $\FRGHT$ and $\mathcal{GH}$ of rooted Greg hypertrees. We have:
    $$ \FRGHT=\frac{1}{u}E_{\geq 1}(u.\mathcal{GH}) $$
    and:
    $$ \mathcal{GH}=X.E(\FRGHT) + vE_{\geq 2}(\FRHT) $$
    Hence, we have:
    $$ f_\mathcal{GH}=\frac{\ln(1+u.tf_\FRGHT)}{u}$$
    and:
    $$ \frac{\ln(1+u.tf_\FRGHT)}{u}=t\exp(f_\FRGHT)+v\exp(f_\FRHT)-vf_\FRGHT-v $$
    We get that:
    $$ f_\FRGHT(t,u,v)=\rev\left(\left(\frac{\ln(1+ut)}{u}-v(\exp(t)-t-1) \right) \exp(-t) \right) $$
\end{proof}

Same as in Section~\ref{sec:FRHT}, one can define insertions and show that they define an operad structure on $\FRGHT$.

\begin{dfn}
    Let $S$ and $T$ be two forests of rooted Greg hypertrees and $i$ be a label of a vertex $v_i$ of $S$. Let $B$ be the tree bellow the vertex $v_i$ in $S$, and $C=\{C_1,\dots,C_n\}$ be the set of forests of children of $v_i$ in $S$ such that each rooted Greg hypertrees that are grafted at $v_i$ by the same edge are in the same forest. The \emph{insertion} of $T$ in $S$ at the vertex $i$ denoted $S\circ_i T$ is the formal sum of all possible way to graft the set of forests of rooted Greg hypertrees $C_1,\dots,C_n$ on black or white vertices of $T$ such that each rooted Greg hypertrees of $C_j$ are grafted at $T$ by the same edge, and then grafting the result on the parent of $v_i$ in $B$. If $T$ is a forest it creates a unique hyperedge that connects all its rooted Greg hypertrees to the parent of $v_i$. 
\end{dfn}

The same computations show that:
 
\begin{prp}
    The insertions satisfy the parallel and sequential axioms. Hence, they give a structure of an operad on the species of forest of rooted Greg hypertrees.
\end{prp}

Let us denote:
\[
x_n=\vcenter{\hbox{\begin{tikzpicture}[scale=0.25]     
    \node (R) at (3,-3) {};
    \node (1) at (0, 0) {};
    \node (dots) at (3, 0) {};
    \node (2) at (6, 0) {};
    \draw[thick] (R) -- (1);
    \draw[thick] (R) -- (2);
    \node at (dots) {$\cdots$};
    \draw[circle, fill=white] (R) circle [radius=24pt];
    \draw[circle, fill=white] (1) circle [radius=24pt];
    \draw[circle, fill=white] (2) circle [radius=24pt];
    \node at (R) {\scalebox{1}{$1$}};
    \node at (1) {\scalebox{1}{$2$}};
    \node at (2) {\scalebox{1}{$n$}};
\end{tikzpicture}}}  
\quad ; \quad
c_n=\vcenter{\hbox{\begin{tikzpicture}[scale=0.25]     
    \node (1) at (0, 0) {};
    \node (dots) at (3, 0) {};
    \node (2) at (6, 0) {};
    \node at (dots) {$\cdots$};
    \draw[circle, fill=white] (1) circle [radius=24pt];
    \draw[circle, fill=white] (2) circle [radius=24pt];
    \node at (1) {\scalebox{1}{$1$}};
    \node at (2) {\scalebox{1}{$n$}};
\end{tikzpicture}}}
\quad ; \quad
g_n=\vcenter{\hbox{\begin{tikzpicture}[scale=0.25]     
    \node (R) at (3,-3) {};
    \node (1) at (0, 0) {};
    \node (dots) at (3, 0) {};
    \node (2) at (6, 0) {};
    \draw[thick] (R) -- (1);
    \draw[thick] (R) -- (2);
    \node at (dots) {$\cdots$};
    \draw[circle, fill=black] (R) circle [radius=24pt];
    \draw[circle, fill=white] (1) circle [radius=24pt];
    \draw[circle, fill=white] (2) circle [radius=24pt];
    \node at (1) {\scalebox{1}{$1$}};
    \node at (2) {\scalebox{1}{$n$}};
\end{tikzpicture}}} 
\]

\begin{prp}
    The operad $\FRGHT$ is generated by arity $2$ elements.
\end{prp}

\begin{proof}\label{prp:bin}
    Let $\Pp$ the suboperad of $\FRGHT$ generated by $\FRGHT(2)$, let us prove by induction on the arity that $\Pp=\FRGHT$.
    \begin{itemize}
        \item Base case: by definition $\Pp(2)=\FRGHT(2)$.
        \item Induction step: let $T\in\FRGHT(n)$, if $T=x_n$ or $g_n$ then $T\in\Pp$ since $\greg$ is generated by arity $2$ elements. If $T=c_n$ then $T\in \Pp$ since $\FRHT$ is generated by arity $2$ elements. Else, $T$ can be obtained by inductively composing copies of $x_i$, $c_j$ and $g_k$ at the leaves.
    \end{itemize}
\end{proof}

\begin{thm}
    The operad $\FRGHT$ is isomorphic to $\CG$.
\end{thm}

\begin{proof}
    Computations show that the relations of $\CG$ are satisfied in the operad $\FRGHT$. Hence, we have a morphism $\CG\to\FRGHT$. Since $\FRGHT$ is generated by arity $2$ elements, the morphism is surjective. Moreover, we have $f_\FRGHT(t,1,1)=f_\CG(t)$. The equality of the Hilbert series shows that this morphism is bijective.
\end{proof}

\section{Reduced Rooted Greg hypertrees}\label{sec:red}

As we have seen in Section~\ref{sec:greg}, the link between the operad $\greg_{-1}$ and the operadic twisting of $\PL$ allowed us to prove that $H^*(\greg_{-1})=\Lie$ which is the suboperad of $\PL$ generated by the Lie bracket. To use the same idea for the operad $\CPL$, we would need to define a differential $d$ on $\CG$ such that $d(x)=g$ and $d(c)=0$. However, such a differential would not be compatible with the operad structure since we would have: 
$$d(0) = d(x\circ_1c - c\circ_2x - (c\circ_1x).(2\;3)) = g\circ_1c - c\circ_2g - (c\circ_1g).(2\;3) \neq 0$$
In order to fix this issue, we will need \emph{reduced} version of the operad $\CG$ which will not be Koszul, but on which such a differential can be defined.

\begin{dfn}
    Let us define the \emph{reduced} $\CG$ operad $\RCG$ by $$\RCG=\CG/\langle g\circ_1c - c\circ_2g - (c\circ_1g).(2\;3)\rangle$$
\end{dfn}

Let us describe the underlying species of $\RCG$ as a subspecies of $\FRGHT$. 

Let us study the rewriting rule $g\circ_1c \mapsto c\circ_2g + (c\circ_1g).(2\;3)$ at the level of the forests of rooted Greg hypertrees. It may be written the following way:
\[
    \vcenter{\hbox{\begin{tikzpicture}[scale=0.25]     
        \node (1) at (0, 0) {};
        \node (2) at (3, 0) {};
        \node (3) at (6, 0) {};
        \node (B) at (3, -3) {};
        \draw[thick] (3)--(B);
        \fill[gray] (B.north)--(1.south)--(2.south)--cycle;
        \draw[thick] (B.north)--(1.south);
        \draw[thick] (B.north)--(2.south);
        \draw[thick] (1.south)--(2.south);
        \draw[circle, fill=black] (B) circle [radius=24pt];
        \draw[circle, fill=white] (1) circle [radius=24pt];
        \draw[circle, fill=white] (2) circle [radius=24pt];
        \draw[circle, fill=white] (3) circle [radius=24pt];
        \node at (1) {\scalebox{1}{$1$}};
        \node at (2) {\scalebox{1}{$2$}};
        \node at (3) {\scalebox{1}{$3$}};
    \end{tikzpicture}}} \; \mapsto \;\vcenter{\hbox{\begin{tikzpicture}[scale=0.25]     
        \node (1) at (-1.5, 0) {};
        \node (2) at (3, -3) {};
        \node (3) at (1.5, 0) {};
        \node (B) at (0, -3) {};
        \draw[thick] (3)--(B);
        \draw[thick] (1)--(B);        
        \draw[circle, fill=black] (B) circle [radius=24pt];
        \draw[circle, fill=white] (1) circle [radius=24pt];
        \draw[circle, fill=white] (2) circle [radius=24pt];
        \draw[circle, fill=white] (3) circle [radius=24pt];
        \node at (2) {\scalebox{1}{$1$}};
        \node at (1) {\scalebox{1}{$2$}};
        \node at (3) {\scalebox{1}{$3$}};
    \end{tikzpicture}}} \;+\; \vcenter{\hbox{\begin{tikzpicture}[scale=0.25]     
        \node (1) at (-1.5, 0) {};
        \node (2) at (3, -3) {};
        \node (3) at (1.5, 0) {};
        \node (B) at (0, -3) {};
        \draw[thick] (3)--(B);
        \draw[thick] (1)--(B);        
        \draw[circle, fill=black] (B) circle [radius=24pt];
        \draw[circle, fill=white] (1) circle [radius=24pt];
        \draw[circle, fill=white] (2) circle [radius=24pt];
        \draw[circle, fill=white] (3) circle [radius=24pt];
        \node at (1) {\scalebox{1}{$1$}};
        \node at (2) {\scalebox{1}{$2$}};
        \node at (3) {\scalebox{1}{$3$}};
    \end{tikzpicture}}}
\]
The hyperedge above the black vertex is no longer present in the right-hand side. This lead to the definition of the following species:

\begin{dfn}
    Let $\Red$ be the species of \emph{reduced} forests of rooted Greg hypertrees which is the subspecies of $\FRGHT$ such that black vertices have no incoming hyperedges.  The species $\Red$ inherits the hypertree weight grading and the Greg grading from $\FRGHT$.
\end{dfn}

\begin{dfn}
    The \emph{height} of a forest of rooted Greg hypertrees is the sum over all hyperedges of their hypertree weight times the number of white vertices in the path from this hyperedge to the root. 
\end{dfn}

\begin{prp}
    The following rewriting system on $\FRGHT$ is convergent:
    \[
    \vcenter{\hbox{\begin{tikzpicture}[scale=0.25]     
        \node (1) at (0, 0) {};
        \node (3) at (6, 0) {};
        \node (B) at (3, -3) {};
        \fill[gray] (B.north)--(1.south)--(3.south)--cycle;
        \draw[thick] (B.north)--(1.south);
        \draw[thick] (B.north)--(3.south);
        \draw[thick] (1.south)--(3.south);
        \node (dot) at (3, 0) {$\cdots$};
        \draw[circle, fill=black] (B) circle [radius=24pt];
        \draw[circle, fill=white] (1) circle [radius=24pt];
        \draw[circle, fill=white] (3) circle [radius=24pt];
        \node at (1) {\scalebox{1}{$1$}};
        \node at (3) {\scalebox{1}{$n$}};
    \end{tikzpicture}}} \; \mapsto \;\vcenter{\hbox{\begin{tikzpicture}[scale=0.25]     
        \node (1) at (0, 0) {};
        \node (3) at (6, -3) {};
        \node (B) at (0, -3) {};
        \draw[thick] (1)--(B);        
        \node (dot) at (3, -3) {$\cdots$};
        \draw[circle, fill=black] (B) circle [radius=24pt];
        \draw[circle, fill=white] (1) circle [radius=24pt];
        \draw[circle, fill=white] (3) circle [radius=24pt];
        \node at (1) {\scalebox{1}{$1$}};
        \node at (3) {\scalebox{1}{$n$}};
    \end{tikzpicture}}} \;+\;\cdots\;+\; \vcenter{\hbox{\begin{tikzpicture}[scale=0.25]     
        \node (1) at (0, -3) {};
        \node (3) at (6, 0) {};
        \node (B) at (6, -3) {};
        \draw[thick] (3)--(B);      
        \node (dot) at (3, -3) {$\cdots$};
        \draw[circle, fill=black] (B) circle [radius=24pt];
        \draw[circle, fill=white] (1) circle [radius=24pt];
        \draw[circle, fill=white] (3) circle [radius=24pt];
        \node at (1) {\scalebox{1}{$1$}};
        \node at (3) {\scalebox{1}{$n$}};
    \end{tikzpicture}}}
\]
For readability, other edges of the black vertex are omitted in the picture, however they are present and stay connected to the black vertex.
\end{prp}

\begin{proof}
    First, let us remark that those rewriting rules strictly decrease the height of the forest of rooted Greg hypertrees. Hence, it terminates.
    Let us apply consecutively two rewriting rules. A simple computation shows that the result does not depend on the order of the rewriting rules.
\end{proof}

\begin{crl}
    The species underlying the operad $\RCG$ is $\Red$.
\end{crl}

\begin{proof}
    The rewriting system of the previous proposition gives us a projection of $\FRGHT$ on $\Red$, by applying the rewriting system in any order. Moreover, all those rewriting rules are consequences of the rule $g\circ_1c \mapsto c\circ_2g + (c\circ_1g).(2\;3)$. Hence, $\Red$ is the operad $\FRGHT$ quotiented by the relation $g\circ_1c \mapsto c\circ_2g + (c\circ_1g).(2\;3)$, which is the definition of $\RCG$.
\end{proof}

Let $u$ be a formal variable encoding the hypertree weight grading and $v$ be a formal variable encoding the Greg weight grading. Let us denote $f_\Red(t,u,v)$ the exponential generating series of $\Red$ according to these grading. It means that $f_\Red(t,u,v)=\sum a_{i,j,k}\frac{t^iu^jv^k}{i!}$ where $a_{i,j,k}$ is the number of forests of reduced rooted Greg hypertrees with $i$ white vertices of hypertree weight $j$, and Greg weight $k$.

\begin{prp}\label{prp:EulerRed}
    The exponential generating series of $\Red$ is given by:
    $$ f_\Red(t,u,v)=\rev\left(\left((v+1)\frac{\ln(1+ut)}{u}+v-v\exp\left(\frac{\ln(1+ut)}{u}\right) \right) \exp(-t) \right) $$
    In particular, $f_\Red(t,1,-1)$ is the series $\sum_{n\geq 1}n^{n-1}\frac{t^n}{n!}$.
\end{prp}

\begin{proof}
    Let us inspect the species $\Red$ and $\RG$ of reduced rooted Greg hypertrees. We have:
    $$ \Red=\frac{1}{u}E_{\geq 1}(u.\RG) $$
    and:
    $$ \RG=X.E(\Red) + vE_{\geq 2}(\RG) $$
    Hence, we have:
    $$ f_{\RG}=\frac{\ln(1+u.tf_\Red)}{u}$$
    and:
    $$ \frac{\ln(1+u.tf_\Red)}{u}=t\exp(f_\Red)+v\exp\left(\frac{\ln(1+u.tf_\Red)}{u}\right)-v\frac{\ln(1+u.tf_\Red)}{u}-v $$
    We get that:
    $$ f_\Red(t,u,v)=\rev\left(\left((v+1)\frac{\ln(1+ut)}{u}+v-v\exp\left(\frac{\ln(1+ut)}{u}\right) \right) \exp(-t) \right) $$
\end{proof}

\begin{rmk}
    The computation of the composition reverse of this exponential generating series show that $\RCG$ is not Koszul.
\end{rmk}

We can now define the analogue of $\greg_{-1}$ for the operad $\CPL$. We will compute its cohomology in the next section to show the main theorem.

\begin{dfn}
    Let $\dg\CG$ be the differential graded operad such that the underlying operad is $\RCG$ with $x$, $y$ and $c$ in degree $0$ and $g$ in degree $1$, and the differential is given by $d(x)=g$ and $d(c)=0$. The underlying species of this operad is $\Red$, hence $d$ is also defined on $\Red$. Let $F^p\Red$ be the subspecies of $\Red$ of forests of reduced rooted Greg hypertrees of height less or equal to $p$. The differential $d$ respect the filtration by the height.
\end{dfn}

\section{The embedding of $\FMan$ in $\CPL$}

The operad $\FMan$ is the operad encoding the algebraic structure on the vector fields of a Frobenius manifold. It is conjectured in \cite{conjFMan} that $\FMan$ is isomorphic to the suboperad of $\CPL$ generated by $x-y$ and $c$. In this section, we will prove this conjecture. First, let us state presentation of the operad $\FMan$ by generators and relations from \cite{HM}. The operad $\FMan$ admit the following presentation:
\begin{multline*}
    \Tt\{l,c\}/\langle l\circ_1 l - l\circ_2 l - (l\circ_1 l).(2\;3),\; c\circ_1 c - c\circ_2 c,\\ (l\circ_1c)\circ_3c - (c\circ_1l)\circ_1c - ((c\circ_1l)\circ_1c).(3\;4) - (c\circ_2l)\circ_3c - ((c\circ_2l)\circ_3c).(1\;2) +\\ ((c\circ_1c)\circ_3l).(2\;3) + ((c\circ_1c)\circ_3l).(1\;3) - ((c\circ_1c)\circ_3l).(1\;4) - ((c\circ_1c)\circ_3l).(2\;4)\rangle, 
\end{multline*}
where the action of $\mathfrak{S}_2$ on $\K\langle l,c\rangle$ is given by $l.(1\;2)=-l$ and $c.(1\;2)=c$. The relations defining $\FMan$ are the Jacobi relation of the Lie bracket $l$, the associativity relation of the commutative product $c$ and the so-called Hertling-Manin relation which is cubical. The Hertling-Manin relation can be understood the following way: Let $LR=l\circ_2c-c\circ_1l - (c\circ_1l).(1\;2)$ be the failure to satisfy the Leibniz rule. Then $LR$ satisfy the Leibniz rule in its first input, meaning that:
$$LR\circ_1c - c\circ_2LR - (c\circ_1LR).(2\;3)=0$$
Since this relation is cubical, $\FMan$ is not quadratic, hence escapes the scope of the Koszul duality theory. However, this operad is closely related to the operad $\PL$, indeed from \cite{conjFMan}, we know that $\FMan$ is the graded operad associated to the filtration of $\PL$ by the embedding of $\Lie$ into $\PL$. In particullar, the arity-wise dimensions of $\FMan$ are the same as the arity-wise dimensions of $\PL$, which are given by the sequence $n^{n-1}$.

In order to prove that we have an embedding of $\FMan$ into $\CPL$, we will compute the cohomology of $\dg\CG$ and show that it is $\FMan$. In order to do so, we will show that the cohomology of $\dg\CG$ is concentrated in degree $0$, then since we know the arity-wise Euler characteristic of $\dg\CG$, we know the arity-wise dimension of the cohomology of $\dg\CG$. Moreover, since those are the same as the arity-wise dimension of $\FMan$, we will have an isomorphism between $\FMan$ and the cohomology of $\dg\CG$, thus showing the embedding of $\FMan$ into $\CPL$.

Let us give a description of the differential $d$ on $\dg\CG$ similar to the description of the differential of $\Tw\PL$ given in Proposition~\ref{prp:descrdGreg}. To do so, let us denote:
\[
x_n=\vcenter{\hbox{\begin{tikzpicture}[scale=0.25]     
    \node (R) at (3,-3) {};
    \node (1) at (0, 0) {};
    \node (dots) at (3, 0) {};
    \node (2) at (6, 0) {};
    \draw[thick] (R)--(1);
    \draw[thick] (R)--(2);
    \node at (dots) {$\cdots$};
    \draw[circle, fill=white] (R) circle [radius=24pt];
    \draw[circle, fill=white] (1) circle [radius=24pt];
    \draw[circle, fill=white] (2) circle [radius=24pt];
    \node at (R) {\scalebox{1}{$1$}};
    \node at (1) {\scalebox{1}{$2$}};
    \node at (2) {\scalebox{1}{$n$}};
\end{tikzpicture}}} 
\quad ; \quad
g_n=\vcenter{\hbox{\begin{tikzpicture}[scale=0.25]     
    \node (R) at (3,-3) {};
    \node (1) at (0, 0) {};
    \node (dots) at (3, 0) {};
    \node (2) at (6, 0) {};
    \draw[thick] (R)--(1);
    \draw[thick] (R)--(2);
    \node at (dots) {$\cdots$};
    \draw[circle, fill=black] (R) circle [radius=24pt];
    \draw[circle, fill=white] (1) circle [radius=24pt];
    \draw[circle, fill=white] (2) circle [radius=24pt];
    \node at (1) {\scalebox{1}{$1$}};
    \node at (2) {\scalebox{1}{$n$}};
\end{tikzpicture}}}
\quad ; \quad
c_n=\vcenter{\hbox{\begin{tikzpicture}[scale=0.25]     
    \node (1) at (0, 0) {};
    \node (dots) at (3, 0) {};
    \node (2) at (6, 0) {};
    \node at (dots) {$\cdots$};
    \draw[circle, fill=white] (1) circle [radius=24pt];
    \draw[circle, fill=white] (2) circle [radius=24pt];
    \node at (1) {\scalebox{1}{$1$}};
    \node at (2) {\scalebox{1}{$n$}};
\end{tikzpicture}}}
\quad ; \quad
h_n=\vcenter{\hbox{\begin{tikzpicture}[scale=0.25]     
    \node (R) at (3,-3) {};
    \node (1) at (0, 0) {};
    \node (dots) at (3, 0) {};
    \node (2) at (6, 0) {};
    \fill[gray] (R.north)--(1.south)--(2.south)--cycle;
    \draw[thick] (R.north)--(1.south);
    \draw[thick] (R.north)--(2.south);
    \draw[thick] (1.south)--(2.south);
    \node at (dots) {$\cdots$};
    \draw[circle, fill=white] (R) circle [radius=24pt];
    \draw[circle, fill=white] (1) circle [radius=24pt];
    \draw[circle, fill=white] (2) circle [radius=24pt];
    \node at (R) {\scalebox{1}{$1$}};
    \node at (1) {\scalebox{1}{$2$}};
    \node at (2) {\scalebox{1}{$n$}};
\end{tikzpicture}}}
\]
Moreover, for $P=\{\{\lambda_{1,1},\dots,\lambda_{1,n_1}\},\dots,\{\lambda_{k,1},\dots,\lambda_{k,n_k}\}\}$ a partition of $\{2,\dots,n\}$, let us denote:
\[
    p_P=\vcenter{\hbox{\begin{tikzpicture}[scale=0.25]     
        \node (R) at (3,-5) {};
        \node (1) at (-6, 0) {};
        \node (dots) at (-3, 0) {};
        \node (2) at (0, 0) {};
        \node (dots2) at (3, 0) {};
        \node (3) at (6, 0) {};
        \node (dots3) at (9, 0) {};
        \node (4) at (12, 0) {};
        \fill[gray] (R.north)--(1.south)--(2.south)--cycle;
        \draw[thick] (R.north)--(1.south);
        \draw[thick] (R.north)--(2.south);
        \draw[thick] (1.south)--(2.south);
        \fill[gray] (R.north)--(3.south)--(4.south)--cycle;
        \draw[thick] (R.north)--(3.south);
        \draw[thick] (R.north)--(4.south);
        \draw[thick] (3.south)--(4.south);
        \node at (dots) {$\cdots$};
        \node at (dots2) {$\cdots$};
        \node at (dots3) {$\cdots$};
        \draw[circle, fill=white] (R) circle [radius=32pt];
        \draw[circle, fill=white] (1) circle [radius=40pt];
        \draw[circle, fill=white] (2) circle [radius=40pt];
        \draw[circle, fill=white] (3) circle [radius=40pt];
        \draw[circle, fill=white] (4) circle [radius=40pt];
        \node at (R) {\scalebox{1}{$1$}};
        \node at (1) {\scalebox{0.8}{$\lambda_{1,1}$}};
        \node at (2) {\scalebox{0.8}{$\lambda_{1,n_1}$}};
        \node at (3) {\scalebox{0.8}{$\lambda_{k,1}$}};
        \node at (4) {\scalebox{0.8}{$\lambda_{k,n_k}$}};
    \end{tikzpicture}}}
\]
Then any reduced rooted Greg hypertree which is a corolla is $g_n$ or $p_P$ for some $n$ or $P$ up to a permutation of the labels.

\begin{prp}
    Let $T$ be a forest of reduced rooted Greg hypertrees, $i$ the label of a leaf and $C$ a corolla, then $d_0(T\circ_i C)=d_0(T)\circ_i C + (-1)^{|T|} T\circ_i d_0(C)$.
\end{prp}

\begin{proof}
    Let us denote $lwt$ for ``lower weight terms'', meaning forests of reduced rooted Greg hypertrees of lower height. We have:
    \begin{align*}
        d_0(T\circ_i C) &= d(T\circ_i C) + lwt\\
        &= d(T)\circ_i C + (-1)^{|T|} T\circ_i d(C) + lwt\\
        &= d_0(T)\circ_i C + (-1)^{|T|} T\circ_i d_0(C) + lwt
    \end{align*}
    Moreover, since we compose a single reduced rooted Greg hypertree in a leaf of $T$, no rewriting are involved in the composition. Hence,  $d_0(T)\circ_i C$ and $T\circ_i d_0(C)$ have the same height. Hence:
    $$ d_0(T\circ_i C)=d_0(T)\circ_i C + (-1)^{|T|} T\circ_i d_0(C) $$
\end{proof}

\begin{lem}
    The differential $d_0$ on $Gr^p\Red$ admits a description similar to Proposition~\ref{prp:descrdGreg}. The image of a forest of reduced rooted Greg hypertrees $T$ is obtained as the sum of six terms:
    \begin{enumerate}
        \item The sum of all possible ways to split a white vertex of $T$ into a white vertex retaining the label and a black vertex above it and to connect the incoming edges to one of the two vertices (hyperedges cannot be grafted on the black vertex), taken with a minus sign.
        \item The sum of all possible ways to split a white vertex of $T$ into a white vertex retaining the label and a black vertex bellow it and to connect the incoming edges to one of the two vertices (hyperedges cannot be grafted on the black vertex), taken with a minus sign.
        \item The sum of all possible ways to split a black vertex of $T$ into two black vertices and to connect the incoming edges to one of the two vertices, taken with a plus sign.
        \item The sum over all the white vertex directly above a hyperedge to graft this white vertex on top of a new black vertex, and to put this new black vertex in the hyperedge in place of the white vertex, taken with a minus sign.
        \item The sum of all possible ways to graft an additional black leaf to $T$, taken with a plus sign.
        \item The sum of all possible ways graft a tree of $T$ on top of a new black root, taken with a minus sign.
    \end{enumerate}
    In this description, we forbid the grafting of rooted hyperedges on black vertices to ensure that the result is a forest of reduced rooted Greg hypertrees. Some black vertex that are created have zero or one child, however, those terms cancel out in the differential, and we are left with a sum of forests of reduced rooted Greg hypertrees.
\end{lem}

\begin{proof}
    Let us denote $d_\descr$ the map described in the proposition. Let us prove that $d_\descr=d_0$. In order to do so, let us prove that $d_\descr(C)=d_0(C)$ for $C$ a corolla, and then that $d_\descr(T\circ_i C)=d_\descr(T)\circ_i C + (-1)^{|T|} T\circ_i d_\descr(C)$ for $i$ a label of a leaf of $T$ and $C$ a corolla.\\
    Let us first prove that $d_\descr(C)=d_0(C)$ for $C$ a corolla. Let $C$ be a corolla, if $C=x_n$ or $g_n$ (up to a permutation) then $d_\descr(C)=d(C)=d_0(C)$ from Proposition~\ref{prp:descrdGreg}. Else, we have $C=p_P$ for some partition $P$, which allows us to write $C$ as a composition the following way: 
    $$C=(\dots(x_k\circ_{i_1}c_{k_1})\circ_{i_2}\dots)\circ_{i_s}c_{k_s}$$
    Let us compute $d_0(C)$:
    \begin{align*}
        d_0(C)&=d(C)+lwt\\
        &= d((\dots(x_k\circ_{i_1}c_{k_1})\circ_{i_2}\dots)\circ_{i_s}c_{k_s})+lwt\\
        &= ((\dots(d(x_k)\circ_{i_1}c_{k_1})\circ_{i_2}\dots)\circ_{i_s}c_{k_s})+lwt\\
        &= ((\dots(d_\descr(x_k)\circ_{i_1}c_{k_1})\circ_{i_2}\dots)\circ_{i_s}c_{k_s})+lwt
    \end{align*}
    Let $T$ be a reduced rooted Greg hypertree appearing in $d_\descr(x_k)$. To conclude, we need to know of when $((\dots(T\circ_{i_1}c_{k_1})\circ_{i_2}\dots)\circ_{i_s}c_{k_s})$ has the same height as $C$. This is the case if and only no rewritings are involved in the compositions, hence if and only if each $i_j$ is the label of a leaf which is the child of a white vertex. This is exactly the condition that ``hyperedges cannot be grafted on the black vertex'' in the terms from~$(1)$ and~$(2)$. Moreover, since all the new vertices coming from the $c_{k_j}$ are leaves connected by hyperedges, the terms from~$(1)$ compensate with the terms from~$(5)$, and the terms from~$(2)$ compensate with the terms from~$(4)$. Hence, we have that $d_\descr(C)=d_0(C)$.\\
    Let us show that $d_\descr(T\circ_i C)=d_\descr(T)\circ_i C + (-1)^{|T|} T\circ_i d_\descr(C)$ for $i$ a label of a leaf of $T$ and $C$ a corolla. Let us assume that $T$ is not the identity since the result is obvious if $T$ is the identity. The vertex $v$ labeled $i$ is a white leaf which is not the root, hence the contributions of $v$ in the sum come from~$(1)$ and~$(5)$ which compensate, and from~$(2)$ which create a new black vertex bellow it. The only thing that changes for the vertices of $C$, once composed in $T$, is that the root of $C$ will no longer be a root, hence the contribution from~$(6)$ will no longer appear. However, the contribution of $v$ is exactly the missing contribution of $C$ that no longer appears once composed in $T$, hence:
    $$d_\descr(T\circ_i C)=d_\descr(T)\circ_i C + (-1)^{|T|} T\circ_i d_\descr(C)$$
    The sign $(-1)^{|T|}$ comes from the order in which we fill the black vertices, see Remark~\ref{rmk:sign}.\\
    Since any forest of reduced rooted Greg hypertrees can be obtained by inductively composing corollas in leaves, we have that $d_\descr=d_0$.
\end{proof}

Now that we have this description, let us compute the cohomology of $\dg\CG$ using the Kunneth formula and the fact that the cohomology of $\greg_{-1}$ is $\Lie$. To do so, we need a way to ``cut down'' a forest of reduced rooted Greg hypertrees into rooted Greg trees.

\begin{dfn}
    Let $T$ be a reduced rooted Greg hypertree. A \emph{maximal subtree} of $T$ is a rooted tree $M$ such that the vertices of $M$ are a subset of the vertices of $T$, the edges of $M$ are a subset of the edges of $T$, and $M$ is a maximal tree for this property. Since $T$ is rooted, we can see edges of $T$ as being directed, $M$ is rooted in its lower vertex (which is unique). This generalizes naturally for forests of reduced rooted Greg hypertrees.
\end{dfn}

\begin{dfn}
    Let us define the \emph{shape} of a reduced rooted Greg hypertree $T$ as the hypertree obtained from $T$ by replacing each maximal subtree $M_i$ by a corolla with the white vertices of $M_i$ as leaves and a new black vertex as root. The shape of a forest of reduced rooted Greg hypertrees is the forest of the shapes of its elements. Let us denote $\varphi(T)=(S,M_1,\dots,M_k)$ with $S$ the shape of $T$ and $M_1$, \dots, $M_k$ the maximal subtrees of $T$.  
\end{dfn}

\begin{dfn}
    Let $\Part(n)$ be the set of partitions of the set $\underline{n}$, and $P\in\Part(n)$. Let us define $\Ss(P)$ as the span of the shapes of forests of reduced rooted Greg hypertrees with maximal subtrees $M_i$ respectively labeled by $E_i\in P$. Let $\Ss(n)=\bigoplus_{P\in\Part(n)}\Ss(P)$.
\end{dfn}

\begin{rmk}
    From this definition, we have that $\Ss$ is a species, the species of shapes of forests of rooted hypertrees. Its sequence of dimension can be computed, it is Sequence~\oeis{A367753}. See Sequence~\oeis{A367752} for the species of shapes of rooted hypertrees.
\end{rmk}

\begin{prp}
    We have that:
    $$ \FRHT(n)\simeq\bigoplus_{P\in\Part(n)}\left(\bigotimes_{E\in P}\RT(E)\right)\otimes\Ss(P) $$
    $$ \Red(n)\simeq\bigoplus_{P\in\Part(n)}\left(\bigotimes_{E\in P}\RGT(E)\right)\otimes\Ss(P) $$
    The isomorphism is given by $\varphi$, let us denote it $\varphi$ as well. Moreover, we have the following isomorphism of chain complexes:
    $$ (\Red(n),d_0)\simeq\bigoplus_{P\in\Part(n)}\left(\left(\bigotimes_{E\in P}\RGT(E),d\right)\right)\otimes\Ss(P) $$
\end{prp}

\begin{proof}
    Let us build the explicit bijections. Let $S$ be a shape of a forest of rooted hypertrees associated to the partition $P$, and $M_1$, \dots, $M_k$ be rooted trees labeled by $E\in P$. Let us replace each corolla of $S$ by $M_i$ such that the labels of the corolla of $S$ agree with the labels of $M_i$. We get a forest of rooted hypertrees. This is the inverse of the construction we used to define the shape of a forest of rooted hypertrees, hence it is a bijection. Same for the forests of rooted Greg trees.\\
    Let $T$ be a forest of reduced rooted Greg trees, and $\varphi(T)=(S,M_1,\dots,M_k)$. Then from the description of the differential $d_0$ on $\Red$, we have:
    $$d_0(T)=\sum_{i=1}^k \pm \varphi^{-1}(S,M_1,\dots,d_0(M_i),\dots,M_k)$$
    This proves the isomorphism of chain complexes. 
\end{proof}

\begin{rmk}\label{rmk:gact}
    Let $\lambda\vdash n$ and let $\Ss(\lambda)$ the direct sum of $\Ss(P)$ for $P$ a partition of $\underline{n}$ in parts of size $\lambda_i$. We have a right action of the group $\mathfrak{S}_\lambda=\prod\mathfrak{S}_{\lambda_i}$ on $\bigotimes_{i}\RGT(\lambda_i)$, a left action of $\mathfrak{S}_\lambda$, and a right action of $\mathfrak{S}_n$ on $\Ss(\lambda)$. The isomorphism of chain complexes is compatible with those actions, meaning that we have the following isomorphism of $\mathfrak{S}_n$--modules:
    $$ \Red(n)\simeq\bigoplus_{\lambda\vdash n}\left(\bigotimes_{i}\RGT(\lambda_i)\right)\otimes_{\mathfrak{S}_\lambda}\Ss(\lambda) $$
\end{rmk}

We can finally apply the Kunneth formula to show that the cohomology of $\dg\CG$ is concentrated in degree $0$.

\begin{thm}
    The cohomology of the operad $\dg\CG$ is concentrated in degree $0$.
\end{thm}

\begin{proof}
    From the previous proposition, we have that:
    $$ (\Red(n),d_0)\simeq\bigoplus_{P\in\Part(n)}\left(\bigotimes_{E\in P}(\RGT(E),d)\right)\otimes \Ss(P) $$
    Hence, we have:
    $$ H^*(\Red(n),d_0)\simeq\bigoplus_{P\in\Part(n)}\left(\bigotimes_{E\in P}H^*(\RGT(E),d)\right)\otimes \Ss(P) $$
    From \cite[Theorem 5.1]{PLTW}, we have that $H^*(\RGT(E),d)$ is concentrated in degree $0$. Hence, $H^*(\Red(n),d_0)$ is concentrated in degree $0$. The spectral sequence associated to the filtration by the height abuts at the first page, hence the cohomology of $(\Red,d)$ is concentrated in degree $0$.
\end{proof}

\begin{rmk}\label{rmk:descrhred}
    From this proof, we can get the following description of the cohomology of $(\Red,d)$:
    $$ H^*(\Red(n),d)\simeq\bigoplus_{P\in\Part(n)}\left(\bigotimes_{E\in P}\Lie(E)\right)\otimes \Ss(P) $$
    This description could allow us to get a recursive formula for the dimension of $H^*(\Red(n),d_0)$ if the dimensions of $\Ss(P)$ were known.
\end{rmk}

\begin{crl}\label{thm:main}
    The morphism $\FMan\to\CPL$ is injective.
\end{crl}

\begin{proof}
    Since $H^*(\dg\CG)=H^0(\dg\CG)$, we have that $H^0(\dg\CG)$ is the suboperad of $\CPL$ generated by $x-y$ and $c$. Hence, we have a surjective morphism $\FMan\to H^0(\dg\CG)$. We know that $\dim(\FMan(n))=n^{n-1}$ from \cite{conjFMan}, and we have computed the Euler characteristic of $\Red$ in Proposition~\ref{prp:EulerRed}. Since the dimensions are the same, the morphism $\FMan\to H^0(\dg\CG)$ is an isomorphism. Hence, $\varphi$ is injective.
\end{proof}

\begin{crl}\label{thm:HSFM}
    Let $u$ be the additional grading of $\FMan$ by the number of commutative product. The Hilbert series of $\FMan$ is given by:
    $$ f_\FMan(t,u)=f_\Red(t,u,-1)=\rev\left( \left(\exp\left(\frac{\ln(1+ut)}{u}\right)-1\right)\exp(-t)\right) $$
\end{crl}

\begin{rmk}
    Moreover, from Remarks~\ref{rmk:gact}~and~\ref{rmk:descrhred}, and using the same notations, we have the following isomorphism of $\mathfrak{S}_n$--modules:
    $$ \FMan(n)\simeq\bigoplus_{\lambda\vdash n}\left(\bigotimes_{i}\Lie(\lambda_i)\right)\otimes_{\mathfrak{S}_\lambda}\Ss(\lambda) $$
\end{rmk}

\begin{rmk}
    Theorem~3.3 from \cite{LePLe} gives a description of the subspace $\Lie(V)\subseteq\PL(V)$ using constructions similar to the operadic twisting of $\PL$. This description can be understood as a consequence of \cite[Theorem 5.1]{PLTW}. It can be generalized to give description of the subspace $\FMan(V)\subseteq\CPL(V)$ using Theorem~\ref{thm:main}.
\end{rmk}

\section*{Acknowledgements}
I express my deepest gratitude to Vladimir Dotsenko for his help and advice. I would also like to thank Thomas Agugliaro for the discussions we had during the writing of this article.

\vfill

\pagebreak

\section*{Appendix}
\hspace{0pt}
\vfill
    \noindent\begin{minipage}{\textwidth}
        \centering
        \textsc{Supplementary Figure 1.} Rewriting rules of $\CPL$\\
        \vspace{15pt}
        \includegraphics{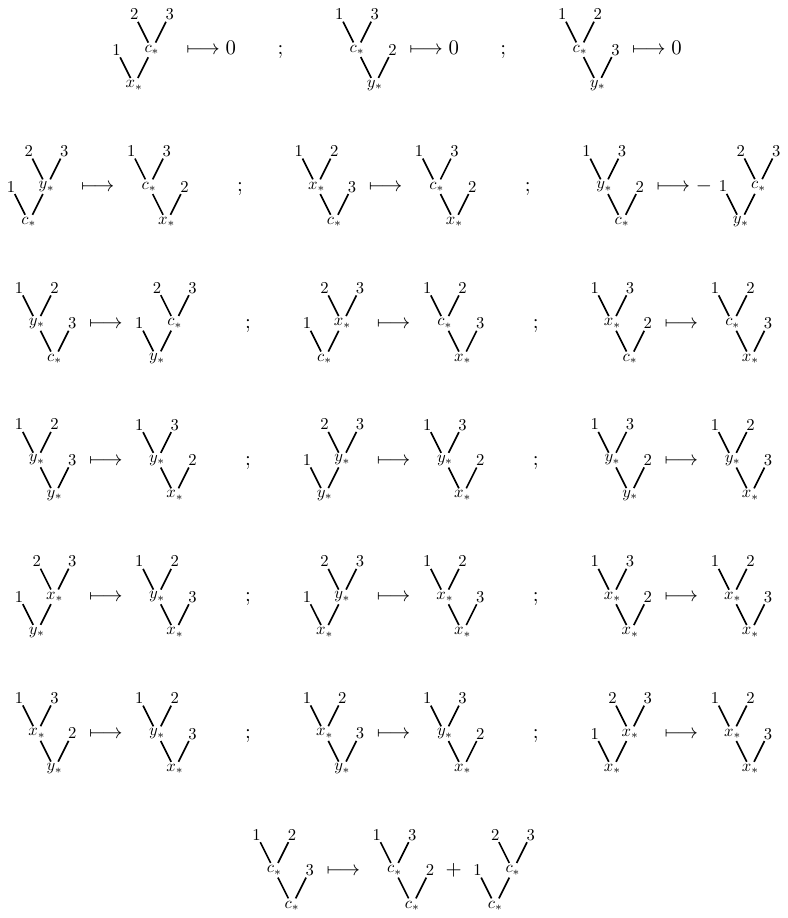}
    \end{minipage}
\vfill
\hspace{0pt}
\pagebreak
\hspace{0pt}
\vfill
    \noindent\begin{minipage}{\textwidth}
        \centering
        \textsc{Supplementary Figure 2.} Rewriting rules of $\CG$ involving $g_*$\\
        \vspace{15pt}
        \includegraphics{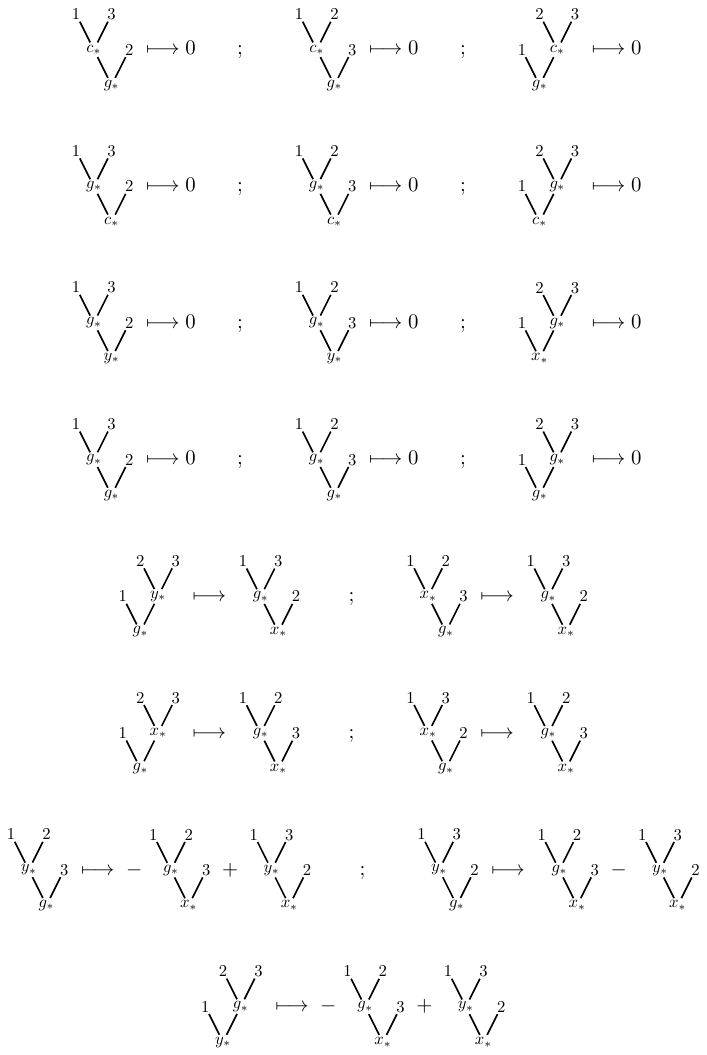}
    \end{minipage}
\vfill
\hspace{0pt}
\pagebreak

\printbibliography

\end{document}